\documentclass[a4paper,12pt]{article}

\usepackage{amssymb}
\usepackage{amsmath, amsthm}

\usepackage[margin=1in]{geometry}

\usepackage{graphicx}

\usepackage{bm}
\usepackage{url}
\usepackage{algorithm}% http://ctan.org/pkg/algorithms
\usepackage{algpseudocode}% http://ctan.org/pkg/algorithmicx
\usepackage{authblk}
\usepackage{comment}
\usepackage{hyperref}

\def\RR{{\bf R}}

\numberwithin{equation}{section}
\newtheorem{thm}{Theorem}[section]
\newtheorem{prop}[thm]{Proposition}
\newtheorem{lem}[thm]{Lemma}

\theoremstyle{definition}

\newtheorem{defn}[thm]{Definition}

\newtheorem{ex}[thm]{Example}

\newcommand*{\joinIrr}{$\vee$-irreducible }
\newcommand*{\joinIrrLast}{$\vee$-irreducible}
\newcommand*{\joinSub}{$(\wedge, \vee)$-closed}
\newcommand*{\LIrr}{$L^{\mathrm{ir}}$ }
\newcommand{\power}[1]{2^{#1}}
\newcommand*{\MnElem}{\mathcal{M}_{n}}
\newcommand{\coll}[3]{C\left(#1,#2,#3\right)}
\newcommand*{\collSym}{C}
\newcommand{\CIde}[1]{\mathcal{C}\left(#1\right)}
\newcommand{\Subsp}[1]{\mathcal{S}\left(#1\right)}
\newcommand{\Csub}[1]{\mathcal{CS}\left(#1\right)}
\newcommand{\PIP}[1]{\mathrm{PIP}\left(#1\right)}
\newcommand{\POS}[1]{\mathrm{PS}\left(#1\right)}
\newcommand{\PPIP}[1]{\mathrm{P}\left(#1\right)}
\newcommand{\app}[2]{#1\left( #2 \right)}
\newcommand{\base}[2]{\bm{e}_{#2}^{#1}}

\newcommand{\subseteqInd}[1]{S^{(#1)}}
\newcommand{\Lirr}[0]{L^{\mathrm{ir}}}
\newcommand{\Birr}[0]{B^{\mathrm{ir}}}

\begin{document}

\title{A Compact Representation \\ for Modular Semilattices and Its Applications}

\author{Hiroshi HIRAI}
\author{So NAKASHIMA}
\affil{%
  Department of Mathematical Informatics,\\%
  Graduate School of Information Science and Technology,\\%
  The University of Tokyo, Tokyo, 113-8656, Japan\\%
  Email: \{\href{hirai@mist.i.u-tokyo.ac.jp}{\texttt{hirai}},%
           \href{so_nakashima@mist.i.u-tokyo.ac.jp}{\texttt{so\_nakashima}}%
         \}\texttt{@mist.i.u-tokyo.ac.jp}%
}

\maketitle

\begin{abstract}
A modular semilattice is 
a semilattice generalization of a modular lattice.
We establish a Birkhoff-type representation theorem 
for modular semilattices, which says that 
every modular semilattice is isomorphic 
to the family of ideals in a certain poset with additional relations.
This new poset structure, which we axiomatize in this paper, 
is called a PPIP (projective poset with inconsistent pairs).
A PPIP is a common generalization of a PIP (poset with inconsistent pairs) and a projective ordered space.
The former was introduced by Barth\'{e}lemy and Constantin for establishing Birkhoff-type theorem for median semilattices, 
and the latter by Herrmann, Pickering, and Roddy for modular lattices.
We show the $\Theta (n)$ representation complexity
and a construction algorithm for 
PPIP-representations of \joinSub{} sets 
in the product $L^n$ of modular semilattice $L$.
This generalizes the results of Hirai and Oki 
for a special median semilattice $S_k$.
We also investigate implicational bases for modular semilattices.
Extending earlier results of Wild and Herrmann for modular lattices, 
we determine optimal implicational bases 
and develop a polynomial time recognition algorithm  for modular semilattices.
These results 
can be applied to retain the minimizer set 
of a submodular function on a modular semilattice.
\end{abstract}


%%%%%%%%%%%%%%%%%%%%%%%%%%%%%%%%%%%%%%%%%%%%%%%%%%%%%%%%
% Text
%%%%%%%%%%%%%%%%%%%%%%%%%%%%%%%%%%%%%%%%%%%%%%%%%%%%%%%%
\section{Introduction}\label{S:intro}

The Birkhoff representation theorem says that
every distributive lattice is isomorphic to the family of ideals 
in a poset (partially ordered set).
This representation of a distributive lattice $L$
is {\em compact} in the sense that 
the cardinality of the poset is at most the height of $L$, and consequently
has brought  numerous algorithmic successes in discrete applied mathematics.
The family of all stable matchings in the stable matching problem
forms a distributive lattice, and is compactly  represented 
by a poset.
Several algorithmic problems on stable matchings 
are elegantly solved by utilizing this poset representation~\cite{gusfield1989stable}.
The family of minimum $s$-$t$ cuts in a network forms a distributive lattice. More generally, the family of minimizers 
of a {\em submodular set function} is a distributive lattice, and 
admits such a compact representation; see~\cite{fujishige2005submodular}.
This fact grew up to theory of {\it principal partitions of submodular systems},
which is a general decomposition paradigm
for graphs, matrices, and matroids~\cite{iri1982structure,fujishige2009theory}.
A canonical block-triangular form of a matrix by means 
of row and column permutations, 
known as the {\em Dulmage-Mendelsohn decomposition 
(DM-decomposition)}, 
is obtained via a maximal chain of 
the family of minimizers of a submodular function, 
in which a maximal chain corresponds 
to a topological order of the poset representation of the family.
The DM-decomposition is further generalized to 
the {\em combinatorial canonical form (CCF)} of  
a {\em mixed matrix} \cite{murota2000matrices, murota1987ccf}, 
which is also built on the same idea.

The present paper addresses Birkhoff-type 
compact representations for lattices and semilattices 
{\em beyond} distributive lattices.
Here, by a compact representation of lattice or semilattice $L$, 
we naively mean a  structure whose size 
is smaller than the size of $L$ 
and from which the original lattice structure can be recovered.
Some of the previous works relating this subject are explained as follows.

{\em Median semilattices} are a semilattice generalization of a distributive lattice, in which every principal ideal is a distributive lattice.
Barth\'{e}lemy and Constantin \cite{barthlemy1993median} 
established a Birkhoff-type representation theorem for median semilattices. Their theorem says that every median semilattice is compactly represented by,  or more specifically, 
is isomorphic to the family of special ideals of a poset with an additional relation, 
called an {\em inconsistency relation}.
This structure is called a {\em poset with inconsistent pairs (PIP)}, which was also independently introduced  by Nielsen, Plotkin, and Winskel \cite{nielsen1981petri}
as a model of concurrency in theoretical computer science, 
and recently rediscovered  by Ardila, Owen, and Sullivant \cite{ardila2012geodesics} 
from the state complex of robot motion planning; 
the name PIP is due to them.
Hirai and Oki~\cite{hirai2016lecture} applied PIP
to represent the minimizer set of a {\em $k$-submodular function}~\cite{huber2012toward}, 
which is a generalization of a submodular set function defined on 
the product ${S_k}^n$ of a special median semilattice 
$S_k$ (consisting of $k+1$ elements). 
They obtained several basic algorithmic results for this PIP-representation. 

{\em Modular lattices} are a well-known lattice class that includes distributive lattices. 
Herrmann, Pickering, and Roddy \cite{herrmann1994a} established
a Birkhoff-type representation theorem for modular lattices, 
which says that
every modular lattice is isomorphic to 
the family of special ideals of a poset with an additional ternary relation, called a {\em collinearity relation}.
This structure is called a \textit{projective ordered space}, and is viewed as a generalization of a {\em projective space}, 
which is a fundamental class of incidence geometries~\cite{ueberberg2011foundation}.

A theory of {\em implicational systems}  (or {\em Horn formulas})~\cite{wild1994theory}
also provides a theoretical basis of compact representations of lattice and semilattice; see recent survey~\cite{wild2016the}. 
Wild~\cite{wild2000optimal}
determines an optimal implicational base 
(or a minimum-size Horn formula) 
of a modular lattice $L$,  where
$L$ is regarded as a {\em closure system} ${\cal F} \subseteq 2^{E}$ on a suitable set $E$.
Consequently, an optimal base is obtained in polynomial time,
when $L$ is given by implications on $\vee$-irreducible elements.
This result is remarkable 
since obtaining an optimal implicational base is NP-hard in general~\cite{maier1980minimum}.
Subsequently,
by utilizing the axiom of projective ordered space, 
Herrman and Wild~\cite{herrmann1996polynomial} developed a polynomial time algorithm to decide whether 
a closure system given by implications is a modular lattice.

The goal of the paper is to generalize these results to {\em modular semilattices}.
This common generalization of median semilattices and modular lattices
was first appeared in Bandelt, van de Vel, and Verheul~\cite{bandelt1993modular}.
Recently,  modular semilattices 
have unexpectedly emerged from
several well-behaved classes of  
combinatorial optimization problems,  and 
been being recognized as 
a next stage on which 
submodular function theory should be developed \cite{Hirai0ext, hirai2016lconvexity}. 
The motivation of this paper 
comes from  these emergences and future contribution of modular semilattices 
in combinatorial optimization.

The results and the organization of this paper are 
outlined as follows:
\paragraph{Section 2:}
We establish a Birkhoff-type representation theorem for modular semilattices: 
Generalizing PIP and projective ordered space, 
we formulate 
the axiom of a new structure {\em PPIP} (projective poset with inconsistency relation), which
 is a certain poset endowed with both inconsistency and collinearity relations.
We prove a one-to-one correspondence 
between modular semilattices and PPIPs (Theorem \ref{theo:BirkhoffSemiModular}).
While projective ordered spaces 
generalize projective geometries, 
PPIP generalizes {\em polar spaces}, which are another fundamental class of incidence geometries.

\paragraph{Section 3:}
A typical emergence of a modular semilattice 
is as a $(\vee, \wedge)$-closed set $B$ in the product $L^n$ 
of a (very small) modular semilattice $L$. 
We investigate the representation complexity of such a modular semilattice $B$. 
We show that 
the number of $\vee$-irreducible elements of $B$ is bounded by $n$ times of the number of \joinIrr elements of $L$ (Theorem \ref{thm:upperBound}).
This attains a lower bound by Berman et al.\cite{berman2009varieties} (Theorem \ref{thm:bermannBoundForSemilattice}), and 
implies that the PPIP-representation for $B$ 
is actually compact (i.e., has a polynomial size in $n$) 
provided the size of $L$ is fixed. 
We give a polynomial time algorithm to construct PPIP 
assuming a membership oracle of $B$ (Theorem \ref{thm:calcPPIP}), which is applied to the minimizer set of a submodular function on $L^n$.
These generalize the results of Hirai and Oki~\cite{hirai2016lecture} 
for the case of $L = S_k$.

\paragraph{Section 4:} Extending Wild's result, we determine 
an optimal implicational base of a modular semilattice 
viewed as a $\cup$-closed family (Theorem \ref{thm:pesudoclosedSetForSemiModular}). 
This naturally leads to a polynomial time algorithm to obtain an optimal
implicational base of a modular semilattice
given by implications on its $\vee$-irreducible elements.
Utilizing the axiom of PPIP, 
we develop a polynomial time recognition algorithm for modular semilattices given by implications (Theorem \ref{thm:RecognizeSemiModular}), which 
is also an extension of the algorithm 
by Herrman and Wild~\cite{herrmann1996polynomial} 
for modular lattices.

\paragraph{Section 5:}
We mention possible applications of these results to (i)
the computation of the PPIP-representation of 
the minimizer set of a submodular function on a modular semilattice and (ii) a canonical block-triangularization of a partitioned matrix~\cite{ito1994blocktriangularizations}, 
which is a further generalization of the DM-decomposition.

\paragraph{Notation}
We use a standard terminology on posets and lattices.
Let $P$ be a poset. 
A subset $X \subseteq P$ is called an \textit{ideal} if $p \leq p'$ and $p' \in X$ implies $p \in X$.
The \textit{principal ideal} of $x$, denoted by $I_x$, is the ideal $\{p \in P \mid p \leq x\}$.
In this paper, semilattices are $\wedge$-semilattices.
Let $L$ be a semilattice. 
Note that the join $x \vee y$ exists if and only if there is a common upper bound of $x$ and $y$.
We say that $l \in L$ is \textit{\joinIrr}if $l = a \vee b$ means $l = a$ or $l = b$.
For a semilattice $L$, let \LIrr{} denote the family of \joinIrr elements of $L$, where
\LIrr{} is regarded as a poset with the partial order derived from $L$.
We denote $\{1,2, \dots, n\}$ by $[n]$. The symbol $|A|$ designates the cardinality of a set $A$. 
\section{Birkhoff-type representation}
\label{sec:3}
In this section, we introduce the concept of PPIP and establish a Birkhoff-type representation theorem for modular semilattices.
After quickly reviewing previous results for median semilattices and modular lattices in Sections~\ref{sec:3.1} and \ref{sec:3.2}, 
we introduce PPIP and give the representation theorem  (Theorem \ref{theo:BirkhoffSemiModular})
in Section~\ref{sec:3.3}. 
The proof of Theorem~\ref{theo:BirkhoffSemiModular} is given in Section~\ref{sec:3.4}.
Throughout this section, 
all posets and semilattices are assumed to have finite rank.

A semilattice $L$ is said to be \textit{modular} \cite{bandelt1993modular} if every principal ideal is a modular lattice, and for every $x,y,z \in L$, the join $x \vee y \vee z$ exists provided $x \vee y$, $y \vee z$, and $z \vee x$ exist.
A \textit{median semilattice} \cite{sholander1954medians} is a modular semilattice each of whose principal ideal is distributive. 
\subsection{Median semilattice and PIP}
\label{sec:3.1}
Here we introduce PIPs and explain a Birkhoff-type representation theorem for median semilattices.
A key tool for the compact representation is a poset endowed with an additional relation.

Let $P$ be a poset. A symmetric binary relation $\smile$ defined on $P$ is called an \textit{inconsistency relation}  if the following conditions are satisfied:
\begin{description}
\item[(IC1)] there are no common upper bounds of $p$ and $q$ provided $p \smile q$;
\item[(IC2)] if $p \smile q$, $p \leq p'$, and $q \leq q'$, then $p' \smile q'$. 
\end{description}
An \textit{inconsistent pair} is a pair $(x,y) \in P^2$ such that $x \smile y$.
A subset $X \subseteq P$ without inconsistent pairs is said to be \textit{consistent}.
\begin{defn}
A \textit{PIP} is a poset endowed with an inconsistency relation.
\end{defn}
Any semilattice $L$ induces an inconsistency relation $\smile$ on the set \LIrr
of \joinIrr elements of $L$.
Define $\smile$ by:  x $\smile$ y holds if and only if $x \vee y$ does not exist.
Then $\smile$ is indeed an inconsistency relation \cite{barthlemy1993median}, and \LIrr becomes a PIP.
To recover $L$ from PIP $\Lirr$,
we introduce the notion of \textit{consistent ideals} in PIP.
Let $P$ be a PIP.
A \textit{consistent ideal} is an ideal without no inconsistent pairs.
Let $\CIde{P}$ denote the family of consistent ideals in $P$.

For a median semilattice $L$, let $\PIP{L}$ denote the PIP on \LIrr with the
induced inconsistency relation $\smile$.
The following theorem establishes Birkhoff-type representation for median semilattices.
\begin{thm}[\cite{barthlemy1993median}]
\begin{itemize}
  \item[{\rm (1)}] Let $L$ be a median semilattice.
        Then $\PIP{L}$ is a PIP, and $\CIde{\PIP{L}}$ is isomorphic to $L$.
  \item[{\rm (2)}] Let $P$ be a PIP. 
        Then $\CIde{P}$ is a median semilattice, and $\PIP{\CIde{P}}$ is isomorphic to $P$.
\end{itemize}
\end{thm}
\subsection{Modular lattice and projective ordered space}
\label{sec:3.2}
We next introduce projective ordered spaces and explain a Birkhoff-type representation theorem for modular lattices.
As in the case of median semilattice, a key tool for the compact representation is a poset endowed with an additional relation.
The axiomatization of projective ordered spaces is necessary to establish our Birkhoff-type representation theorem for modular lattices.

Let $P$ be a poset.
A symmetric ternary relation $\collSym$ defined on $P$ is called a \textit{collinearity relation} \cite{herrmann1994a} if the following conditions are satisfied:
\begin{description}
\item[(CT1)] if $C(p,q,r)$ holds, then $p$, $q$, and $r$ are pairwise incomparable; 
\item[(CT2)] if $\coll{p}{q}{r}$ holds, $p \leq w$, and $q \leq w$, then $r \leq w$.
\end{description}
An \textit{ordered space} is a poset endowed with a collinearity relation.
A triple of elements $x, y, z \in P$ is \textit{collinear} if $C(x,y,z)$ holds.

Any semilattice $L$ induces a collinearity relation on $L^{\mathrm{ir}}$.
Define a ternary relation $C$ on \LIrr by:
$C(x,y,z)$ holds if and only if $x,y,z$ are pairwise incomparable, $x \vee y$, $y \vee z$, $z \vee x$ exist, and $x \vee y = y \vee z = z \vee x$.
Then $C$ is indeed a collinearity relation \cite{herrmann1994a}, and \LIrr becomes an ordered space.
To recover $L$  from $\Lirr$, we introduce the notion of \textit{subspaces} in an
ordered space.
Let $P$ be an ordered space.
An ideal $X \subseteq P$ is called a \textit{subspace} if $p,q \in X$ and the collinearity of $p,q,r$ implies $r \in X$.
Let $\Subsp{P}$ denote the family of subspaces in $P$.

For a modular lattice $L$, let $\POS{L}$ denote the induced ordered space on $\Lirr$.
Then $L$ is isomorphic to $\Subsp{\POS{L}}$ [6]. In particular, any modular
lattice is represented by an ordered space.
However, not all ordered spaces represent modular lattices.
To avoid this inconvenience, Herrmann, Pickering, and Roddy \cite{herrmann1994a} axiomatized projective ordered spaces.

\begin{defn}
\label{defn:POS}
An ordered space $P$ is said to be \textit{projective} if the following axioms are satisfied:
\begin{description}
  \item[(Regularity)] For any collinear triple $(p, q, r)$ and $r' \in P$ such that $r' \leq r$, $r' \not \leq p$, and $r' \not \leq q$, there exist $p' \leq p$ and $q' \leq q$ such that $C(p', q', r')$ holds.
  \item[(Triangle)] If $\coll{a}{c}{p}$ and $\coll{b}{c}{q}$ are satisfied, then at least one of the following conditions holds:
    \begin{itemize}
      \item There exists $x \in P$ such that $C(a,b,x)$ and $C(p,q,x)$ hold, $\{a,b,c,p,q,x\}$ are pairwise incomparable, and there are no collinear triples in $\{a,b,c,p,q,x\}$ other than $(a,c,p)$, $(b,c,q)$, $(a,b,x)$, $(p,q,x)$, and their permutations;
      \item There is $a' \leq a$ such that $\coll{b}{q}{a'}$ holds;
      \item $C(b,q,p)$ holds;
      \item There are $a' \leq a$ and $p' \leq p$ such that $\coll{q}{a'}{p'}$ holds;
      \item $q \leq a$ or $q \leq p$.
    \end{itemize}
\end{description}
\end{defn}

The following theorem establishes a Birkhoff-type representation theorem for modular lattices.
\begin{thm}[\cite{herrmann1994a}]
\label{theo:modular}
\begin{itemize}
\item[{\rm (1)}] Let $L$ be a modular lattice.
Then $\POS{L}$ is a projective ordered space,
and $\CIde{\POS{L}}$ is isomorphic to $L$, where
an isomorphism $\phi \colon L \rightarrow \Subsp{\POS{L}}$ is given by 
\[ \phi(l) := \{p \in \PPIP{L} \mid p \leq l \}. \]
The inverse $\psi$ is given by $\psi(I) := \bigvee_{x \in I} x$
with $\psi(\emptyset) = \min L$.
\item[{\rm (2)}] Let $P$ be a projective ordered space.
Then $\Subsp{P}$ is a modular lattice,
and $\POS{\Subsp{P}}$ is isomorphic to $P$.
\end{itemize}
\end{thm}
\subsection{Modular semilattice and PPIP}
\label{sec:3.3}
Modular semilattices are a common generalization of median
semilattices and modular lattices.
Therefore one may expect that modular semilattices are represented
by a structure generalizing PIP and projective ordered space.
Here we introduce such a structure, named a PPIP, and
establish a Birkhoff-type representation theorem for modular semilattices.

\begin{defn}
Let $P$ be a poset associated with an inconsistency relation $\smile$ and collinearity relation $\collSym$. We say that $P$ is a \textit{PPIP} if the following axioms are satisfied:
\begin{description}
  \item[(Regularity)] The same as in Definition \ref{defn:POS}.
  \item[(weak Triangle)] Suppose that $\coll{a}{c}{p}$ and $\coll{b}{c}{q}$ hold and $\{a,b,c,p,q\}$ is consistent.
    Then at least one of the five conditions of Triangle axiom in Definition \ref{defn:POS} holds.
  \item[(Consistent-Collinearity)] For any collinear triple $(p,q,r)$, the following conditions are satisfied:
  \begin{description}
    \item[(CC1)] the set $\{p,q,r \}$ is consistent;
    \item[(CC2)] for any $x \in P$, the element $x$ is consistent with either at most one of $(p,q,r)$ or all of them.
  \end{description}
\end{description}
\end{defn}

For a modular semilattice $L$, let $\PPIP{L}$ denote \LIrr{} equipped with the induced inconsistency relation and collinearity relation. 
We will later prove that $\PPIP{L}$ is a PPIP if $L$ is a modular semilattice. 
For a PPIP $P$, let $\Csub{P}$ be the family of consistent subspaces of PPIP $P$.
Regard $\Csub{P}$ as a poset with respect to the inclusion order $\subseteq$.

The main result in this section is the following:
\begin{thm}
\label{theo:BirkhoffSemiModular}
  \begin{description}
    \item[{\rm (1)}]Let $L$ be a modular semilattice. 
      Then $\PPIP{L}$ is a PPIP,
      and $\Csub{\PPIP{L}}$ is isomorphic to $L$, where 
      an isomorphism $\phi \colon L \rightarrow \Csub{\PPIP{L}}$ is given by 
\[\phi(l) := \{p \in \PPIP{L} \mid p \leq l \}.\] 
     The inverse $\psi$ is given by $\psi(I) := \bigvee_{x \in I} x$ with  
$\psi(\emptyset) = \min L$.
    \item[{\rm (2)}] Let $P$ be a PPIP.
      Then $\Csub{P}$ is a modular semilattice,
      and $\PPIP{\Csub{P}}$ is isomorphic to $P$. 
  \end{description}
\end{thm}
In particular, a modular semilattice is compactly represented by a PPIP.
This theorem will be proved in the next section.
\begin{ex}
A modular semilattice, illustrated in Figure \ref{fig:PPIPrep} (a), is represented by the PPIP in Figure \ref{fig:PPIPrep} (b).
\end{ex}

\begin{figure}[t]
  \begin{center}
    \begin{tabular}{c}

      % 1
      \begin{minipage}{0.5\hsize}
        \begin{center}
          \includegraphics[clip, width=6.5cm]{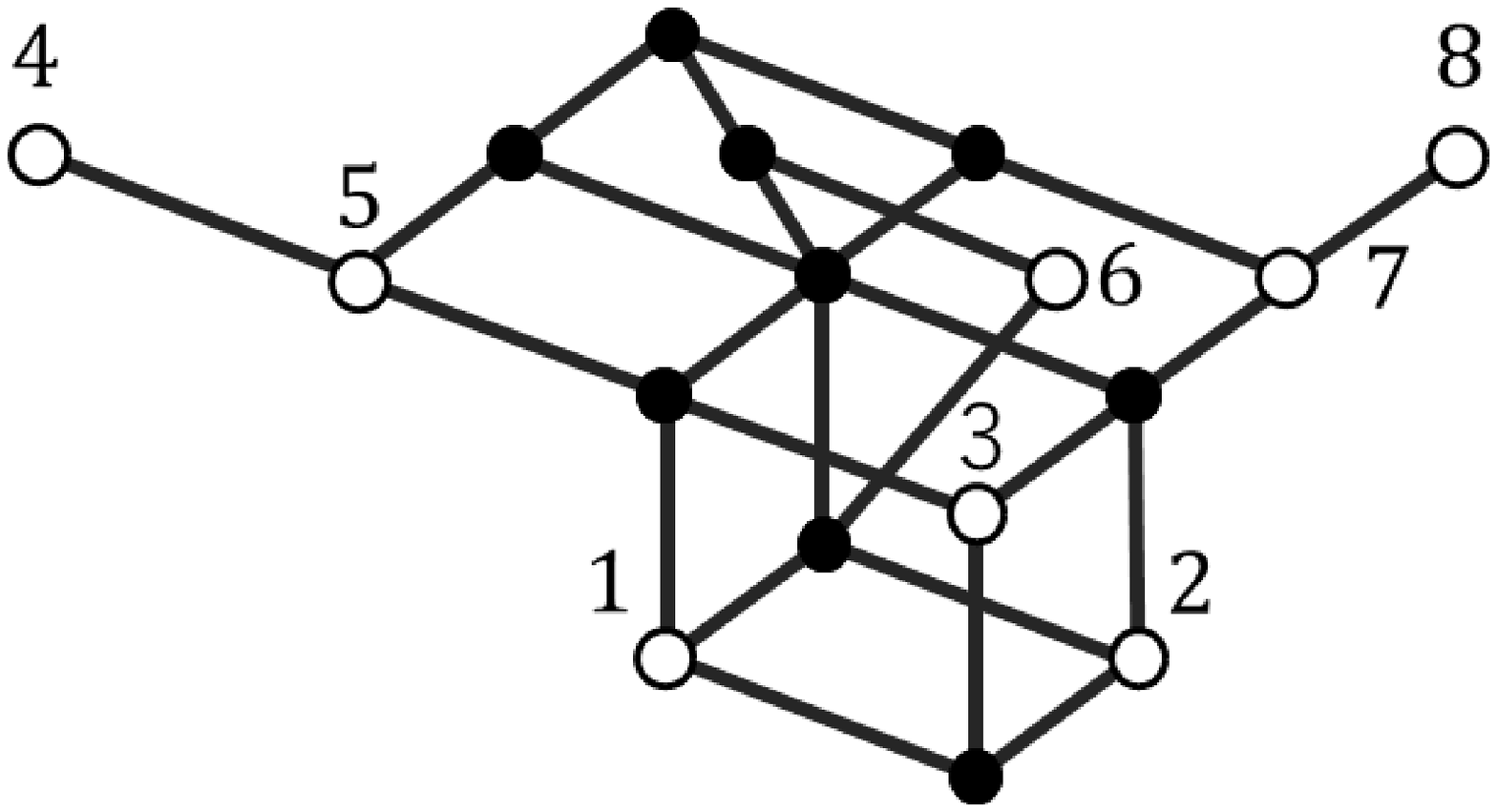}
          \hspace{1.6cm} (a)
        \end{center}
      \end{minipage}

      % 2
      \begin{minipage}{0.5\hsize}
        \begin{center}
          \includegraphics[clip, width=6.5cm]{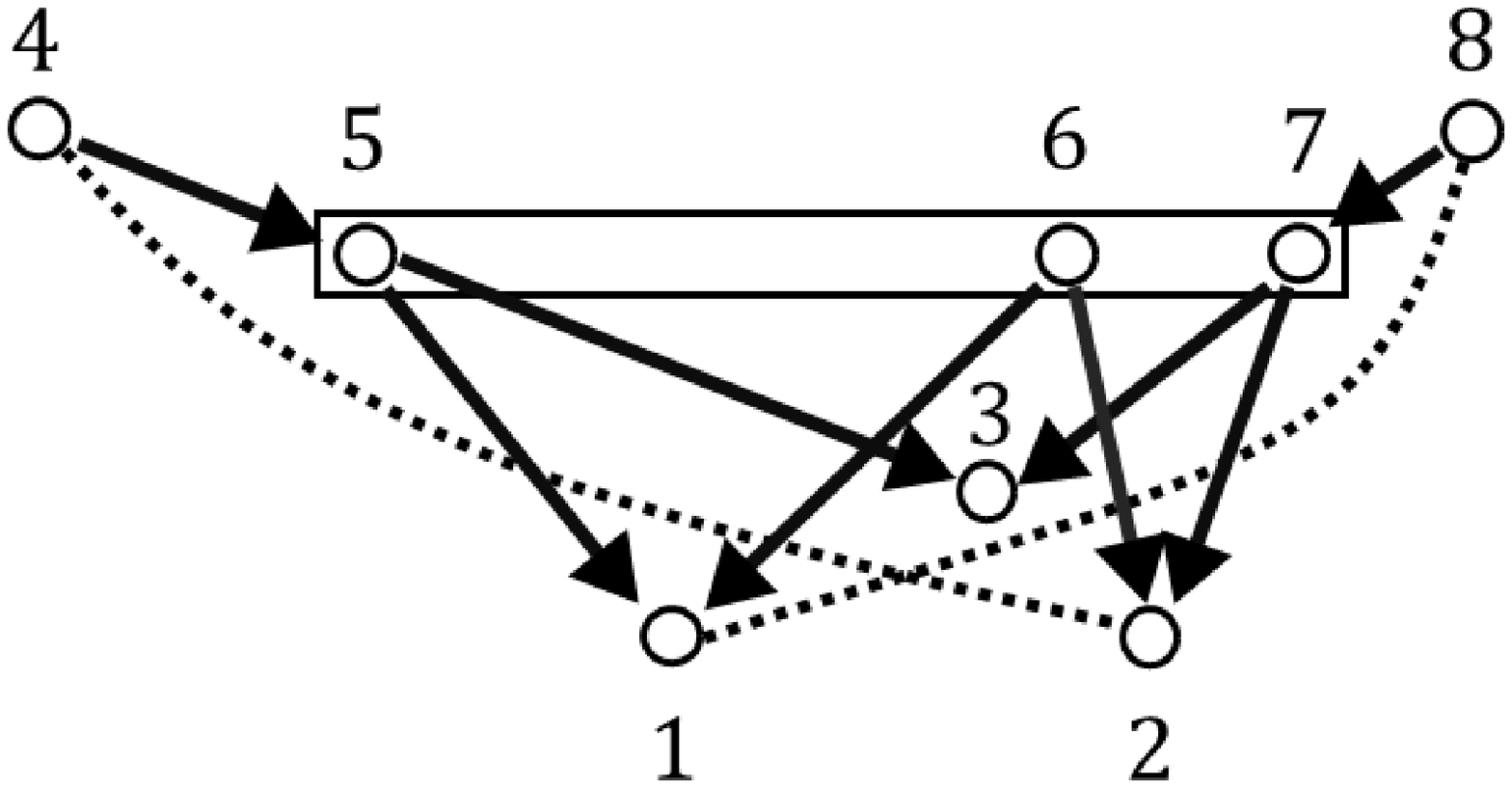}
          \hspace{1.6cm} (b)
        \end{center}
      \end{minipage}

    \end{tabular}
    \caption{An example of PPIP representation. (a) Hasse diagram of a modular semilattice. Its \joinIrr elements are represented by white dots and numbered. (b) PPIP representation of the modular semilattice. Dots and arrows constitute its Hasse diagram. Dotted line represents minimal inconsistent pairs (defined in Section \ref{sec:4.1}). Three elements in the rectangular box are collinear.  }
    \label{fig:PPIPrep}
  \end{center}
\end{figure}
\begin{ex}
A PPIP can be viewed as a generalization of a \textit{polar space}~\cite{ueberberg2011foundation}.
A \textit{point-line geometry} is a pair $(P,L)$ of a set $P$ and $L \subseteq \power{P}$,
 where an element in $P$ is called a \textit{point} and an element in $L$ is called a \textit{line}. 
We say that a line $l$ connects $p$ and $q$ and that $p$ and $q$ are on $l$ if $p, q \in l$.
Two points $p$ and $q$ are said to be \textit{collinear} if there is a line $l$ connecting $p$ and $q$.
The point-line geometry $(P,L)$ is called a \textit{polar space} if the following conditions are satisfied:
\begin{itemize}
  \item[(i)] for any points $p,q$, there is at most one line connecting $p$ and $q$;
  \item[(ii)] for any line $l$, there are at least 3 points on $l$;
  \item[(iii)] for any line $l$ and a point $p$, there exist either exactly one point on $l$ collinear with $p$, or all points on $l$ are collinear with $p$.
\end{itemize}

Any polar space $(P,L)$ is a PPIP.
Indeed, we regard $P$ as a poset each pair of whose elements is incomparable.
We define an inconsistency relation on $P$ by $p \smile q$ if and only if $p$ is not collinear with $q$, and a collinearity relation by $C(p,q,r)$ holds if and only if $p$, $q$, and $r$ are on a common line.
Then it is clear that $P$ satisfies Consistent-Collinearity axiom.
That $P$ satisfies Regularity and weak Triangle axioms follows from the fact that every \textit{subspace} of a polar space is a \textit{projective space} \cite{ueberberg2011foundation}.

A canonical example of polar spaces is the family of \textit{totally isotropic} subspaces in vector space $V$ with nondegenerate alternating bilinear form $B$.
A subspace $W \subseteq V$ is said to be \textit{totally isotropic} if $B(W, W) = \{0\}$. 
A polar space in Figure \ref{fig:polar} corresponds to the case where $V = \mathrm{GF}(2)^3$ and $B$ is identified with a matrix 
\[
B =  \left(
    \begin{array}{ccc}
      0 &  1 & 1   \\
      1 & 0 &  1 \\
      1 & 1 & 0\\
    \end{array}
  \right).
\]
Each point corresponds to a subspace spanned by each of vectors
\[ (1,1,1), (0,0,1), (1,1,0), (0,1,0), (1,0,1), (1,0,0), (0,1,1),\]
in the numerical order.
\end{ex}
\begin{figure}[t]
  \centering
  \includegraphics[clip,width=4.5cm]{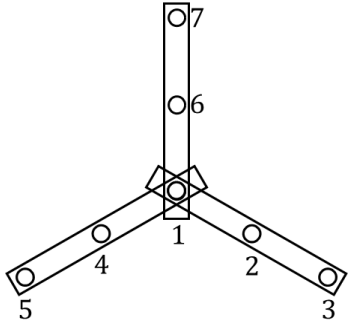}
  \caption{A polar space of totally isotropic spaces in $\mathrm{GF}(2)^3$.}
  \label{fig:polar}
\end{figure}
\subsection{Proof of Theorem \ref{theo:BirkhoffSemiModular}}
\label{sec:3.4}
\begin{lem}
\label{lem:inductiveConstruction}
Let $P$ be a PPIP and $X$ its consistent subset.
Then there exists a consistent subspace $S \in \Csub{P}$ such that $X \subseteq S$.
\end{lem}
\begin{proof}
We construct $S$ inductively:
\begin{align*}
&S^{(0)} := X,\\
&S^{(2n+1)} := S^{(2n)} \cup f\left(S^{(2n)}\right), \\
&S^{(2n + 2)} := S^{(2n+1)} \cup g\left(S^{(2n+1)}\right),\\
&S := \bigcup_n S^{(n)},
\end{align*}
where $f(W) := \{p \in P \mid \exists p', p'' \in W \;\; C(p,p',p'') \text{ holds} \} $ and $g(W) := \{p \in P \mid \exists p' \in W \;\; p \leq p' \}$.
It is easy to see that $S$ is a subspace and $X \subseteq S$.

Hence it suffices to prove that $S$ is consistent.
By assumption, $S^{(0)}$ is consistent.
Suppose by induction that $S^{(2n)}$ is consistent. 
We show that $S^{(2n+1)}$ is consistent.
Let $x, y \in S^{(2n+1)}$.
Then one of the following cases holds:
(i) $x,y \in S^{(2n)}$; 
(ii) $x \in S^{(2n+1)} \setminus S^{(2n)}$ and $y \in S^{(2n)}$;
(iii) $x, y \in S^{(2n+1)} \setminus S^{(2n)}$.
In case (i), $x \not \smile y$ by induction.
In case (ii), there are $x_1, x_2 \in S^{(2n)}$ such that $C(x, x_1, x_2)$ holds. 
Since $y \in S^{(2n)}$ is consistent with both $x_1$ and $x_2$ by induction, (CC2) implies $x \not \smile y$.
In case (iii),
there are $y_1, y_2 \in S^{(2n)}$ such that $C(y,y_1,y_2)$ holds.
Then $x$ is consistent with both $y_1$ and $y_2$ by the case (ii).
By using (CC2) again, we see that $x \not \smile y$. 
Thus $S^{(2n+1)}$ is consistent.
Next we show that $S^{(2n+2)}$ is also consistent. 
Suppose to the contrary that there is an inconsistent pair $x \smile y \in S^{(2n+2)}$.
By the definition of $S^{(2n+2)}$, there exist $x', y' \in S^{(2n+1)}$ such that $x \leq x'$ and $y \leq y'$. 
Then (IC2) implies $x' \smile y'$.
This contradicts the fact that $S^{(2n+1)}$ is consistent.
Hence $S^{(2n+2)}$ is consistent.
\end{proof}
\begin{prop}
\label{prop:PPIPtoModularSemi}
Let $P$ be a PPIP. 
Then $\Csub{P}$ is a modular semilattice.
\end{prop}
\begin{proof}
It is obvious that $\Csub{P}$ is a semilattice whose $\wedge$ is set intersection $\cap$.
We prove that every principal ideal of $\Csub{P}$ is a modular lattice.
Let $X \in \Csub{P}$. %
Regard $X$ as a subPPIP of $P$.
Since $X$ is consistent, weak Triangle axiom is equivalent to Triangle axiom in $X$. %
In particular, $X$ is a projective ordered space.
By Theorem \ref{theo:modular} and the fact $I_X = \Subsp{X}$, the principal ideal $I_X$ is a modular lattice.

We next show that $X \vee Y \vee Z$ exists provided  $X \vee Y$, $Y \vee Z$, and $Z \vee X$ exist. 
By Lemma \ref{lem:inductiveConstruction}, it suffices to show that $X \cup Y \cup Z$ is consistent. 
This follows from the existence of $X \vee Y$, $Y \vee Z$, and $Z \vee X$.
\end{proof}
\begin{prop}
Let $L$ be a modular semilattice.
Then $\PPIP{L}$ is a PPIP.
\end{prop}
\begin{proof}
Regularity and weak Triangle axiom are shown by the restriction of $L$ to its appropriate principal ideal.
Suppose that the premise of the Regularity axiom holds.
By the definition of the induced collinearity relation, $\{p,q,r\}$ are consistent. 
In particular, $l := p \vee q \vee r$ exist.
The restriction of $\PPIP{L}$ to $\{p \in \PPIP{L} \mid p \leq l\}$, written as $\PPIP{L} \restriction_l$, is isomorphic to $\POS{L}$ as an ordered space.
Hence $\PPIP{L} \restriction_l$ is projective.
Notice that $\PPIP{L} \restriction_l$ contains $p,q,r,r'$ in Regularity axiom.
By Regularity axiom of $\PPIP{L} \restriction_l$, we obtain $p' \leq p$ and $q' \leq q$ such that $C(p', q', r')$ holds.
Weak Triangle axiom is shown by the restriction to $\{p \in \PPIP{L} \mid p \leq l'\}$, where $l' = a \vee b \vee c \vee p \vee q$.

Next we prove Consistent-Collinearity axiom.
The condition (CC1) follows by definition of the induced collinearity relation. 
Suppose to the contrary that (CC2) is not true.
Then there exist $x,p,q,r \in \PPIP{L}$ such that $x \smile p$,  $x$ is consistent with $q$ and $r$, and $C(p,q,r)$ holds.
By the consistency of $\{x,q,r\}$, the join $l := x \vee q \vee r$ exists.
Since $C(p,q,r)$ holds, $p \leq p \vee q = q \vee r \leq l$.
In particular, $l$ is a common upper bound of $x$ and $p$.
This contradicts $x \smile p$.
\end{proof}
\begin{prop}
\label{prop:SemiModularRepMain}
Let $L$ be a modular semilattice. 
Then $L$ is isomorphic to $\Csub{\PPIP{L}}$. 
An isomorphism $\phi \colon L \rightarrow \Csub{\PPIP{L}}$ is given by 
$\phi(l) := \{p \in \PPIP{L} \mid p \leq l \}$. The inverse $\psi$ is given by $\psi(I) := \bigvee_{x \in I} x$.
Here $\psi(\emptyset) = \min L$.
\end{prop}
\begin{proof}
First we show that both $\phi$ and $\psi$ are well-defined order preserving maps. 
By the consistency, $\psi$ is well-defined.
It is easy to show that both $\phi$ and $\psi$ preserve the partial orders.
Let us check that $\phi(l)$ is indeed a consistent subspace.
It is trivial that $\phi(l)$ is an ideal.
Since $l$ is a common upper bound of $\phi(l)$, every pair in $\phi(l)$ has its join in $L$.
In particular, $\phi(l)$ is consistent.
Suppose $(p,q,r)$ is a collinear triple and $p,q, \in \phi(l)$. 
By the definition of the induced collinearity relation, $r \leq r \vee p = p \vee q \leq l$.
Hence $r \in \phi(l)$.
This means that $\phi(l)$ is a subspace.

Next we prove $\psi \circ \phi$ is the identity map.
Clearly, $\app{\psi}{\app{\phi}{\min L}} = \min L$.
Let $l \in L \setminus \{ \min L \}$.
Since $l$ is a common upper bound of $\phi(l)$, we have $\app{\psi}{\app{\phi}{l}} = \bigvee_{x \in \phi(l)} x \leq l$.
By the finite length condition of $L$, $l$ is decomposed into \joinIrr elements $\{p_i\}$ as $l = p_1 \vee p_2 \vee \dots \vee p_n$.
Since $\{p_i\} \subseteq \phi(l)$, we have $l = p_1 \vee p_2 \vee \dots \vee p_n \leq \bigvee \phi(l)$.
Hence $\app{\psi}{\app{\phi}{l}} = l$.

Finally we show that $\phi \circ \psi$ is the identity map. 
Let $X \in \Csub{\PPIP{L}}$. 
We prove that $X = \app{\phi}{\app{\psi}{X}}$ by restricting $L$ to the principal ideal $I_{\psi(X)}$.
Notice that $\POS{I_{\app{\psi}{X}}}$ is isomorphic to the restriction of $\PPIP{L}$ to $\app{\phi}{\app{\psi}{X}}$ as an ordered space.
We can easily check that $X$ is a subspace of $\POS{I_{\app{\psi}{X}}}$.
The restriction of $\phi$ and $\psi$ to $I_{\psi(X)}$ and $\app{\phi}{\app{\psi}{X}}$ respectively, is the same as in Theorem \ref{theo:modular}.
Hence $X = \app{\phi}{\app{\psi}{X}}$ follows from Theorem \ref{theo:modular}.
\end{proof}
Thus we completed the proof of Theorem \ref{theo:BirkhoffSemiModular} (1).
Next we prove (2).
\begin{lem}
\label{lem:joinSemiModular}
Let $P$ be a PPIP. 
For any $S,T \in \Csub{P}$, if $S \cup T$ is consistent, then the join of $S,T$ exists in $\Csub{P}$ and is given by
\begin{equation}
\label{eq:join}
S \vee T = S \cup T \cup \{r \in P \mid \exists p \in S, q \in T \text{ such that } C(p,q,r) \text{ holds} \}.
\end{equation}
\end{lem}
\begin{proof}
By Lemma \ref{lem:inductiveConstruction}, there exists a common upper bound $U \in \Csub{P}$ of $S$ and $T$.
In particular, $S \vee T$ exists.
Since $U$ is consistent, $U$ can be regarded as a projective ordered space and $\Csub{U}$ is identical with $\Subsp{U}$.
It was shown in \cite[LEMMA 3.1]{herrmann1994a} that equation (\ref{eq:join}) holds for any subspaces $S,T$ in a projective ordered spaces.
Thus, $S \vee T$ is given by (\ref{eq:join}).
\end{proof}
\begin{lem}
\label{lem:joinIrrInPPIP}
Let $P$ be a PPIP.
Then the family of \joinIrr elements of $\Csub{P}$ is equal to that of principal ideals of $P$. 
\end{lem}
\begin{proof}
We first prove that any principal ideal is a \joinIrrLast{} element.
Let $p \in P$. 
The principal ideal $I_p$ of $p$ is clearly a consistent subspace.
Suppose that there exist $X, Y \in \Csub{P}$ such that $I_p = X \vee Y$.
Then $p \in X \vee Y$.
By Lemma \ref{lem:joinSemiModular}, one of the following conditions holds: 
$p \in X$ or $p \in Y$;
there exist $x \in X$ and $y \in Y$ such that $C(p,x,y)$ holds. 
In the first case, we can assume $p \in X$ without loss of generality.
Then $I_p \subseteq X$ since $X$ is an ideal.
This implies $I_p = X$.
In the second case, $p$ and $x$ are incomparable by (CT1).
In particular, $x \not \in I_p$.
Hence $I_p \subsetneq X \vee Y$.
This contradicts $I_P = X \vee Y$.
We have thus proved that principal ideal is \joinIrrLast.

We next show that every \joinIrr element is a principal ideal.
Let $X$ be a \joinIrr element of $\Csub{P}$.
Obviously $X$ is written as $X = \bigvee_{x \in X} I_{x}$.
By the finite length condition of $P$, there is a finite subset $X' \subseteq X$ such that $X = \bigvee_{x \in X'} I_{x}$.
Since $X$ is \joinIrrLast, X equals one of the components in the right-hand side. 
This means that $X$ is a principal ideal.
\end{proof}
\begin{prop}
Let $P$ be a PPIP.
Then $P$ is isomorphic to $\PPIP{\Csub{P}}$, an isomorphism $f \colon P \rightarrow \PPIP{\Csub{P}} $ is given by $\phi(p) = I_p$.
\end{prop}
\begin{proof}
By Lemma \ref{lem:joinIrrInPPIP}, the function $f$ is a well-defined bijection. 
In the rest of the proof, we check that $f$ preserves all relations on $P$.
Clearly $f$ preserves the partial order of $P$.

We next prove that $f$ preserves the inconsistency relation of $P$. 
Let $(p, q)$ be an inconsistent pair in $P$.
Then there are no consistent subspaces including $I_p \cup I_q$.
Hence $I_p \vee I_q$ does not exist.
Thus $I_p$ and $I_q$ are inconsistent in $\Csub{P}$.
Conversely, suppose that $I_p$ and $I_q$ are inconsistent in $\PPIP{\Csub{P}}$.
Since there are no consistent subspaces including $I_p$ and $I_q$, the union $I_p \cup I_q$ contains an inconsistent pair $(x , y)$.
By (IC1), we can assume that $x \in I_p$ and $y \in I_q$.
Then $p \smile q$ follows from (IC2).

We finally show that $f$ preserves the collinearity relation.
Let $(p,q,r)$ be a collinear triple in $P$.
Since $p$, $q$, and $r$ are pairwise incomparable by (CT1), so are $I_p$, $I_q$, and $I_r$.
If the existence of $I_q \vee I_r$ and $I_p, I_q \subseteq I_q \vee I_r$ are proved, 
then the existence of $I_q \vee I_r$ and $I_r \vee I_p$, and $I_p \vee I_q = I_q \vee I_r = I_r \vee I_p$ follow from symmetry, that is, $C(I_p, I_q, I_r)$ holds.
The condition (CC1) implies $q \not \smile r$.
By (IC2), the union $I_q \cup I_r$ is also consistent.
By Lemme \ref{lem:inductiveConstruction}, $I_q \vee I_r$ exist.
Since $C(p,q,r)$ holds, it must hold that $p \in I_q \vee I_r$ by Lemma \ref{lem:joinSemiModular}.
Then $I_p \subseteq I_q \vee I_r$ because $I_q \vee I_r$ is an ideal.
We have thus shown that $I_p, I_q \subseteq I_q \vee I_r$ as required.

Conversely, suppose $C(I_p, I_q, I_r)$ holds.
Since $I_r \subseteq I_r \vee I_q = I_p \vee I_q$, we have $r \in I_p \vee I_q$.
By Lemme \ref{lem:joinSemiModular}, one of the following conditions holds:
(i) $r \in I_p$ or $r \in I_q$;
(ii) there are $p' \in I_p$ and $q' \in I_q$ such that $C(p',q',r)$ holds.
Condition (i) contradict to the incomparability of $I_p, I_q, I_r$.
Hence condition (ii) is true.
Since we have proved that the collinearity of $I_p,I_q,I_r$ follows from that of $p, q, r$, condition (ii) implies that $C(I_{p'}, I_{q'}, I_r)$ holds.
Then
\[
I_p \vee I_q = I_q \vee I_r \subseteq I_q \vee (I_r \vee I_{q'}) = I_q \vee (I_{p'} \vee I_{q'}) = I_{p'} \vee I_q.
\]
Here we used the collinearity of $I_p, I_q, I_r$ and $I_{p'}, I_{q'}, I_r$, and $q' \leq q$.
This inequality implies $p \in I_{p'} \vee I_{q}$.
By using Lemma \ref{lem:joinSemiModular} again, one of the following conditions hold: 
(i) $p \in I_{p'}$;
(ii) $p \in I_{q}$;
(iii) there are $p'' \in I_{p'}$ and $q'' \in I_{q}$ such that $C(p,p',q'')$ holds.
Condition (ii) contradicts the incomparability of $p$ and $q$.
Since $p'' \leq p' \leq p$, condition (iii) contradicts (CT1).
Hence Condition (i) holds, which means $p = p'$.
We can prove $q = q'$ by the same argument.
We have thus proved that $C(p,q,r)$ holds.
\end{proof}
\section{{\joinSub{}} set in {$L^n$}}
\label{sec:5}
A modular semilattice typically arises as a \joinSub{} set 
of the $n$-product $L^n$ of some small modular semilattice $L$.
In this section, we investigate computational and algorithmic aspects on PPIP-representations  of \joinSub{} sets in $L^n$. 
In Section \ref{sec:5.1}, 
we show that any \joinSub{} set in $L^n$ 
admits a PPIP of size polynomial 
in $n$ and $|L^{\rm ir}|$ (Theorem \ref{thm:upperBound}).
In Section \ref{sec:6.2}, we present a polynomial time algorithm to compute the PPIP-representation of a \joinSub{} set in $L^n$ using Membership Oracle.
In this section, all semilattices are assumed to be finite.
\subsection{$O(n|L^{\mathrm{ir}}|)$-bound of \joinIrr elements}
\label{sec:5.1}
In this section, we show the $\Theta(n)$ representations complexity of \joinSub{} sets in $L^n$.
Let $L$ be a semilattice.
The symbol $L^n$ denotes an $n$-product of $L$, whose partial order is the product order.
Notice that we can compute $\wedge$ and $\vee$ of $L^n$ in the component-wise manner, that is, the following identity holds for any $\bm{l} = (l_1, l_2, \dots, l_n) \in L^n$ and $\bm{l}' = (l_1', l_2', \dots, l_n')$:
\begin{align*}
&\bm{l} \wedge \bm{l'} = (l_1 \wedge l_1', l_2 \wedge l_2', \dots, l_n \wedge l_n'),\\
&\bm{l} \vee \bm{l'} = (l_1 \vee l_1', l_2 \vee l_2', \dots, l_n \vee l_n') \text{ ($\bm{l} \vee \bm{l'}$ exists if all $l_i \vee l_i'$ exist).}
\end{align*}
A subset $B \subseteq L^n$ is said to be \textit{\joinSub{}} if $b_1 \wedge b_2 \in B$ for any $b_1, b_2 \in B$ and $b_1 \vee b_2 \in B$ for any $b_1, b_2 \in B$ such that $b_1 \vee b_2$ exists in $L$.
If $L$ is a modular semilattice, then $L^n$ and $B$ are modular semilattice.
In the following, let $L$ be a semilattice and $B$ a \joinSub{} set in $L^n$ without further mentioning.

Our compact representation theorems are valid for $B$ if $L$ is a modular semilattice.
However computational problems still remain.
As the cardinality of $L^n$ grows exponentially, so may that of $\PPIP{B}$. 
Moreover, it is unrealistic enumerating \joinIrr elements of $B$ by a brute-force search.
Hirai and Oki \cite{hirai2016lecture} solved these problems for \joinSub{} sets of ${S_k}^n$, where $S_k$ is a $k+1$ element semilattice such 
that elements other than the minimum element are pairwise incomparable,

We generalize this result to \joinSub{} sets of arbitrary semilattices. 
In this section, we give the upper bound of $\PPIP{B}$.
The enumerating problem will be treated in the next section. %
We owe the following theorem and its proof to a discussion with Taihei Oki.
\begin{thm}
\label{thm:upperBound}
Let $L$ be a semilattice and $B$ a \joinSub{} set in $L^n$.
The cardinality of \joinIrr elements of $B$ is at most $n|\Lirr|$.
\end{thm}
\begin{proof}
It suffices to show that $|\Birr| \leq |\Lirr|$ for any semilattice $L$ and its \joinSub{} set $B$ since the cardinality of \joinIrr of $L^n$ is $n|\Lirr|$.
We prove this claim by constructing an injection $g \colon \Birr \rightarrow \Lirr$.
A function $\phi \colon L \rightarrow \power{\Lirr}$ is defined by $\phi(l) := \{p \in \Lirr \mid p \leq l\}$, and a partial function $\psi \colon \power{\Lirr} \rightarrow L$ is defined by $\psi(I) = \bigvee I$.
Notice that we can prove that $\psi \circ \phi$ is an identity map for arbitrary semilattice $L$ in the same way as Proposition \ref{prop:SemiModularRepMain}.
For any a \joinIrr element $b \in B$, let $\underline{b}$ be the unique lower cover.
Then $g(b)$ is defined as an arbitrary element in $\phi(b) \setminus \phi(\underline{b})$.
We first show that this definition of $g$ is well-defined.
Since $\psi \circ \phi$ is an identity map, $\psi$ is injective.
Hence $\phi(b) \setminus \phi(\underline{b})$ is not empty and $g$ is well-defined.
We next show that $g$ is injective.
Let $b_1, b_2$ be distinct elements in $\Birr$.
We prove $g(b_1) \neq g(b_2)$ by the case analysis for the comparability of $b_1$ and $b_2$.
Suppose that $b_1$ and $b_2$ are comparable.
We can assume that $b_2 \leq b_1$.
Since 
$g(b_1) \in \phi(b_1) \setminus \phi(\underline{b_1})$, 
$g(b_2) \in \phi(b_2)$, and
$\phi(b_2) \subseteq \phi(\underline{b_1})$,
we have $g(b_1) \neq g(b_2)$.
Suppose that $b_1$ and $b_2$ are incomparable.
Then $b_1 \wedge b_2$ equals $\underline{b_1} \wedge \underline{b_2}$.
Suppose to the contrary that $g(b_1) = g(b_2)$.
Then $g(b_1) \in (\phi(b_1) \cap \phi(b_2) ) \setminus (\phi(\underline{b_1}) \cup \phi(\underline{b_2}) )$.
Since 
\[
\phi(b_1) \cap \phi(b_2) = \phi(b_1 \wedge b_2) = 
\phi(\underline{b_1} \wedge \underline{b_2}) \subseteq \phi(\underline{b_1}) \cup \phi(\underline{b_2}),
\]
the set $(\phi(b_1) \cap \phi(b_2) ) \setminus (\phi(\underline{b_1}) \cup \phi(\underline{b_2}) )$ is empty.
This is a contradiction.
We have thus proved $g(b_1) \neq g(b_2)$.
\end{proof}
Our compact representation achieves the lower bound given by Berman et al. \cite{berman2009varieties}.
Berman et al. regarded \textit{generating sets} as compact representations of \textit{algebras}.
Then they characterized how fast the size of generating sets of \textit{subalgebra} in the $n$-product of an algebra grows with respect to $n$ by \textit{$k$-edge terms}.
We apply their result to \joinSub{} sets in products of semilattices and show that the lower bound of such compact representations is $\Omega(n)$. 
In particular, our compact representation is optimal, that is, achieves the lower bound.

We briefly review the results of Berman et al. 
They used \textit{universal algebra} to achieve their goal. 
See \cite{berman2009varieties} for more details, and \cite{sankappanavar1981course} for universal algebra.
An \textit{algebra} is a set endowed with \textit{fundamental operations}, where each fundamental operation is a function $A^k \rightarrow A$.
In the following, we regard semilattices as algebras whose fundamental operations are $\wedge$ and $\sqcup$, where $a \sqcup b$ equals $a \vee b$ if $a \vee b$ exists and to $a \wedge b$ otherwise.
Let $A$ be an algebra.
A \textit{subalgebra} of $A$ is a subset $B \subseteq A$ that is closed under all fundmental operations.
We remark that subalgebras of a semilattice are \joinSub{} sets.
Let $B$ be a subalgebra of $A$.
A \textit{generating set} of $B$ is a subset $C \subseteq B$ such that we can obtain all elements in $B$  by applying fundamental operations to those in $C$ iteratively.
Note that the \joinIrr elements of \joinSub{} set forms a generating set.
The $n$-product of $A$ is the algebra whose underlying set is $A^n$ and fundamental operations act in a component-wise manner.
Notice that the $n$-product of semilattices is the $n$-product in a universal algebraic sense.

Let $B$ be a subalgebra of $A^n$.
Berman et al. regarded a generating set of $B$ as a compact representation of $B$, and characterized how fast the minimum cardinality of a generating set grows with respect to $n$.
Indeed, a generating set contains all information about $B$, and can be regarded as a compact representation.
Let us define a function $g$ which measures the size of such representations.
Let $A$ be an algebra, $B$ a subalgebra of $A^n$, and $g_A^B(n)$ the minimum of the cardinalities of generating sets of $B$.
We define $g_A(n)$ by $g_A(n) = \max_{B} g_A^B(n)$, where $B$ runs over all subalgebras of $A^n$.

\begin{thm}
\label{thm:bermannBoundForSemilattice}
There is a semilattice $L$ such that $g_L(n) = \Omega(n)$.
\end{thm}

Therefore the general lower bound of such a compact representation for $B$ is $\Omega(n)$.
Since our compact representation is a part of that of Berman et al., it achieves the lower bound and is optimal.
The rest of this section is devoted to proving Theorem \ref{thm:bermannBoundForSemilattice}.

\begin{defn}
We refer $(k+1)$-ary operation $t$ satisfying the following equations as a $k$-edge term.
\begin{align*}
t(y,y,x,x,x, \dots,x) &= x,\\
t(y,x,y,x,x, \dots,x) &= x,\\
t(x,x,x,y,x,\dots, x) &= x,\\
t(x,x,x,x,y, \dots,x) &= x,\\
&\vdots\\
t(x,x,x,x,x, \dots, y) &= x.\\
\end{align*}
Let $A$ be an algebra. We say that $A$ has a $k$-edge term if a $k$-edge term is obtained by composing fundamental operations iteratively.
\end{defn}
\begin{thm}[\cite{berman2009varieties}]
Assume that $A$ has a $k$-edge term but no $l$-edge terms for $l < k$.
Then $g_A(n) = \Omega(n^{k-2})$.
\end{thm}

Semilattices have a 3-edge term but not generally have a 2-edge term.
Indeed, $t(x,y,z,w) = (y \wedge z) \sqcup (z \wedge w) \sqcup (w \wedge y)$ is a 3-edge term.
However, a chain $0 < a < b < 1$ is a semilattice but has no 2-edge terms.
We remark that a 2-edge term is a \textit{Mal'cev term} and that an algebra has a Mal'cev term if and only if it is \textit{congruence-permutable} \cite[Chapter 2, Theorem 12.2]{sankappanavar1981course}.
Since this chain is not congruence-permutable, it does not have a 2-edge term.
We used a discussion in Math StackExchange (\nolinkurl{http://math.stackexchange.com/}\nolinkurl{questions/1809557/not-congruence-permutable-lattice}) as a reference.
Thus we have proved Theorem \ref{thm:bermannBoundForSemilattice}.
\subsection{Constructing PPIP from Menbership Oracle}
\label{sec:6.2}
In this section, we present an algorithm to compute the PPIP-representation of \joinSub{} set $B$ in $L^n$ from Membership Oracle.
We first characterize \joinIrr elements of $B$.
Then we address an algorithm enumerating \joinIrr elements of $B$ via the characterization.
After that we present an algorithm which computes the PPIP-representation of $B$ in $O(n^3|\Lirr|^3 + n^2|L|^2)$-time.
In the following, let $L$ be a modular semilattice.

We first introduce the concept of \textit{bases} to characterize \joinIrr elements of $B$.
Note that $B$ is a modular semilattice.
For $\bm{l} = (l_1, l_2, \dots, l_n) \in L^n$, we denote $l_i$ by $\bm{l}[i]$.
A \textit{base} of $B$ is an element of the form
$\bm{e}^i_l = \min \{\bm{b} \in B \mid \bm{b}[i] = l\}$
for some $i \in [n]$ and $l \in L$. 
If the set in the right-hand side is empty, the corresponding base is undefined.
The concept of base was introduced by Hirai and Oki \cite{hirai2016lecture} for ${S_k}^n$.

The following lemma plays an important role not only in proving theorems but also constructing algorithms.
We can decide whether $\base{i}{l} \leq \bm{b}$ in $O(1)$-time while a
naive algorithm takes $O(n)$-time provided two elements in $L$ are compared in $O(1)$-time.
\begin{lem}[LCP: Locally Comparing Property]
\label{lem:LCP}
Assume that $\base{i}{l}$ exists. 
Then $\base{i}{l} \leq \bm{b}$ if and only if $l \leq \bm{b}[i]$.
\end{lem}
\begin{proof}[Proof of lemma \ref{lem:LCP}]
By the definition of product order, $\base{i}{l} \leq \bm{b}$ implies $l \leq \bm{b}[i]$ for all $\bm{b} \in B$.
Conversely, suppose that $l \leq \bm{b}[i]$.
Then $(\bm{b} \wedge \base{i}{l})[i] = l$.
The minimality of $\base{i}{l}$ implies $\base{i}{l} \leq \bm{b} \wedge \base{i}{l} \leq \bm{b}$.
\end{proof}
Now we are ready to characterize \joinIrr elements of $B$.
Let $\pi_i(B) := \{\bm{b}[i] \mid \bm{b} \in B\}$.
Note that $\pi_i(B)$ is a modular semilattice.
\begin{thm}
\label{thm:characterizeJoinIrr}
An element $\bm{b} \in B$ is \joinIrr if and only if $\bm{b}$ is equal to $\bm{e}^i_l$ for some $i \in [n]$ and $l$ a \joinIrr element in $\pi_i(B)$.
\end{thm}
\begin{proof}
We first show that bases of such a form are \joinIrrLast.
Suppose that $\base{i}{l}$, where $l$ is \joinIrr in $\pi_i(B)$, is decomposed as $\base{i}{l} = \bm{b}_1 \vee \bm{b}_2$.
Then $l = \bm{b}_1[i] \vee \bm{b}_2[i]$.
Since $l$ is \joinIrr in $\pi_i(B)$, we can assume that $\bm{b}_1[i] = l$.
By LCP, $\base{i}{l} \leq \bm{b}_1$.
Consequently $\base{i}{l} = \bm{b}_1$ and $\base{i}{l}$ is \joinIrrLast.

We next show that \joinIrr elements of $B$ are bases of such a form.
Assume that $\bm{b} \in B$ is \joinIrrLast.
It obviously holds that
\[
\bm{b} = \base{1}{\bm{b}[1]} \vee \base{2}{\bm{b}[2]} \vee \dots \vee \base{n}{\bm{b}[n]}.
\]
Since $\bm{b}$ is \joinIrrLast, $\bm{b}$ is equal to one of the $\{\base{i}{\bm{b}[n]}\}$.
In particular, $\bm{b}$ is a base.
Therefore it suffices to show that $\base{i}{l}$ is not \joinIrr if so is not $l$ in $\pi_i(B)$.
Suppose that $l$ is not \joinIrrLast, that is, $l = s \vee t$ for $s,t \in L \setminus \{l\}$.
We prove that $\base{i}{l} = \base{i}{s} \vee \base{i}{t}$.
This implies that $\base{i}{l}$ is not $\vee$-irreducible.
By LCP, $\base{i}{l}$ is greater that $\base{i}{s}$ and $\base{i}{t}$.
As a result, $\base{i}{l} \geq \base{i}{s} \vee \base{i}{t}$.
Conversely, $(\base{i}{s} \vee \base{i}{t})[i] = l$.
Using LCP again, we have $\base{i}{l} \leq \base{i}{s} \vee \base{i}{t}$.
\end{proof}

We present an efficient algorithm to compute the PPIP-representation of $B$.
We first show that we can enumerate \joinIrr elements of \joinSub{} sets in $L^n$ by at most $n^2|L|^2$ calls of \textit{Membership Oracle}.
Then we construct an algorithm to compute PPIP representation in $O(n^3|L|^3)$-time.

We first enumerate \joinIrr elements of $B$ under the assumption where Membership Oracle (MO) is available.
An important example of MO is a minimizer oracle. 
We later show that the minimizer set of a submodular function on $L^n$ forms a \joinSub{} set, and that MO of the minimizer set is obtained from  a minimizer oracle. 
In this sense, it is a natural assumption that MO is available.
\begin{defn}
Membership Oracle (MO) for a \joinSub{} set $B \subseteq L^n$ answers the following decision problem:
\begin{description}
  \item[Input:] $i,j \in [n]$, $l,l' \in L$, 
  \item[Output:] Whether or not there exists $\bm{b} \in B$ such that $\bm{b}[i] = l$ and $\bm{b}[j] = l'$.
\end{description}
\end{defn}
\begin{thm}
Suppose that MO is available. 
Then all bases of $B$ are obtained by at most $n^2|L|^2$ calls of MO. 
\end{thm}
\begin{proof}
It suffices to show that the component $\base{i}{l}[j]$ is computed by at most $|L|$ calls of MO since there are at most $n|L|$ bases. 
The case where $j = i$ is obvious.
Otherwise compute the set $S^i_l[j] := \{l' \in L \mid \exists \bm{b} \in B \text{ such that } \bm{b}[i] = l \text{ and } \bm{b}[j] = l'\}$ by $|L|$ calls of MO with input $(i, j, l, l')$ for each $l' \in L$.
Then we obtain $\base{i}{l}[j]$ as $\min S^i_l[j]$.
If $S^i_l[j]$ is empty, then $\base{i}{l}$ is undefined.
\end{proof}
Thus we can enumerate all \joinIrr elements.
It suffices to enumerate all bases and check whether each enumerated base is \joinIrr by Theorem \ref{thm:characterizeJoinIrr}.

We finally construct an efficient algorithm to compute $P(B)$.
Assume that, for any $l_1, l_2 \in L$, the order, meet, and join of $l_1, l_2$ are computed in $O(1)$-time throughout the rest of this section.
This assumption is justified when $|L|$ is very small compared to $n$. 
\begin{thm}
\label{thm:calcPPIP}
The PPIP-representation $P(B)$ can be obtained in $O(n^3|\Lirr|^3 + n^2|L|^2)$-time; the algorithm enumerates the partial order, inconsistency relation, and collinearity relation of $P(B)$.
\end{thm}
This algorithm is $O(n|L|)$-times faster than a brute-force search.
In the rest of this section, we show Theorem \ref{thm:calcPPIP}.

We can construct the Hasse diagram of $P(B)$ from 
enumerated \joinIrr elements in $O(n^2|L|^2)$-time.
Also we associate each element $\bm{a}$ in $P(B)$ 
with $E(\bm{a}) 
:= \{(i,l) \in [n] \times L  \mid  \bm{a} = \base{i}{l} \}$ (by a list).
Notice that we can compare any bases in $O(1)$-time by LCP.
We next deal with inconsistent pairs.

\begin{prop}
All inconsistent pairs can be computed in $O(n^2|L|^2)$-time.
\end{prop}
\begin{proof}
We first show that $\bm{a}, \bm{b} \in P(B)$ are inconsistent if and only if there exist $\base{i}{l}, \base{i}{l'} \in \PPIP{B}$ such that $l \smile l'$, $\base{i}{l} \leq \bm{a}$, and $\base{i}{l'} \leq \bm{b}$.
The if-part is obvious. 
Conversely, suppose that $\bm{a} \smile \bm{b}$.
Since $\bm{a}, \bm{b}$ are inconsistent, the join of $\bm{a}, \bm{b}$ does not exist.
Hence there is $i \in [n]$ such that $\bm{a}[i] \vee \bm{b}[i]$ does not exist.
Then, by Theorem \ref{theo:BirkhoffSemiModular}, there are inconsistent \joinIrr elements $l, l' \in \pi_i(B)$ such that $l \leq \bm{a}[i]$ and $l' \leq \bm{b}[i]$.
By LCP, $\bm{e}^i_l \leq \bm{a}$ and $\bm{e}^i_{l'} \leq \bm{b}$.
By Theorem \ref{thm:characterizeJoinIrr}, $\bm{e}^i_l, \bm{e}^i_{l'}$ are \joinIrrLast{} elements.
Thus we have proved that $\bm{a} \smile \bm{b}$ implies the existence of $\base{i}{l}, \base{i}{l'} \in \PPIP{B}$ such that $l \smile l'$, $\base{i}{l} \leq \bm{a}$, and $\base{i}{l'} \leq \bm{b}$.

By using this characterization, we can enumerate all inconsistent pairs as follows:
(i) enumerate pairs of \joinIrr bases of the form $\bm{e}^i_l, \bm{e}^i_{l'}$ with $l \smile l'$;
(ii) enumerate pair $\bm{a}, \bm{b}$ of $\PPIP{B}$ such that there is a pair $\bm{e}^i_l, \bm{e}^i_{l'}$ enumerated in (i) with $\bm{e}^i_l \leq \bm{a}$ and $\bm{e}^i_{l'} \leq \bm{b}$.

The first step (i) can be done by checking lists $E(\bm{a})$ and $E(\bm{b})$ for all $\bm{a}, \bm{b} \in P(B)$.
The second step (ii) can be done 
by searching  the Hasse diagram of poset $P(B) \times P(B)$
from pairs obtained in (i).
The whole procedure can be done in $O(n^2 |L|^2)$-time.
\end{proof}

We finally enumerate collinear triples.
\begin{prop}
\label{prop:characterizationOfCollinearity}
Suppose that $\base{i}{a}$, $\base{j}{b}$, and $\base{k}{c}$ are \joinIrr and pairwise consistent.
Then the following conditions are equivalent:
\begin{description}
  \item[(i)] $C(\base{i}{a}, \base{j}{b}, \base{k}{c})$ holds.
  \item[(ii)] $C(a, \base{j}{b}[i], \base{k}{c}[i])$ holds in $\pi_i(B)$, $C(\base{i}{a}[j], b, \base{k}{c}[j])$ in $\pi_j(B)$, and $C(\base{i}{a}[k], \base{j}{b}[k], c)$ in $\pi_k(B)$.
\end{description}
\end{prop}
\begin{proof}
First we prove that (i) implies (ii).
It suffices to show that $C(a, \base{j}{b}[i], \base{k}{c}[i])$ holds in $\pi_i(B)$.
For convenience, let $\alpha = a$, $\beta := \bm{e}^j_b[i]$, and $\gamma := \bm{e}^k_c[i]$.
It follows from (CT2) and the collinearity of $(\base{i}{a}, \base{j}{b}, \base{k}{c})$ that $\alpha \vee \beta = \beta \vee \gamma = \gamma \vee \alpha$.
We next show that $\alpha$, $\beta$, and $\gamma$ are pairwise incomparable.
The incomparability implies that $C(a, \base{j}{b}, \base{k}{c})$ holds.
Suppose that $\alpha \leq \beta$ or $\alpha \leq \gamma$.
By LCP, it contradicts to the incomparability of $\base{i}{a}$, $\base{j}{b},$ and $\base{k}{c}$.
Suppose that $\alpha \geq \beta$ or $\alpha \geq \gamma$.
We may deal with the case $\alpha \geq \beta$.
Then $\beta \vee \gamma = \alpha \vee \beta = \alpha$.
Since $\base{i}{a}$ is \joinIrrLast, Theorem \ref{thm:characterizeJoinIrr} implies $\alpha$ is \joinIrr in $\pi_i(B)$.
Hence, $\beta = \alpha$ or $\gamma = \alpha$, both of which contradicts to the incomparability of  $\base{i}{a}$, $\base{j}{b}$, and $\base{k}{c}$ by LCP.
Suppose that $\beta$ and $\gamma$ are comparable.
We may assume that $\beta \leq \gamma$.
Then $\gamma \vee \alpha = \beta \vee \gamma = \gamma$.
In particular, $\alpha \leq \gamma$, which contradicts to the incomparability of $\base{i}{a}$ and $\base{k}{c}$ by LCP.
We have thus completed the proof of the incomparability of $\base{i}{a}$, $\base{j}{b}$, and $\base{k}{c}$. 

Next we prove that (ii) implies (i).
Since $\base{i}{a}$, $\base{j}{b}$, and $\base{k}{c}$ are consistent, $\bm{l} = \base{i}{a} \vee \base{j}{b} \vee \base{k}{c}$ exists.
By (CT1), $\base{i}{a}[x]$, $\base{j}{b}[x]$, and $\base{k}{c}[x]$ are pairwise incomparable for $x = i,j,k$.
Using LCP, we have the incomparability of $\base{i}{a}$, $\base{j}{b}$, and $\base{k}{c}$.
Therefore it suffices to show $\base{i}{a} \vee \base{j}{b} = \base{j}{b} \vee \base{k}{c} = \base{k}{c} \vee \base{i}{a}$.
If the equality
\[
\base{i}{a} \vee \base{j}{b} = \mathrm{min}\{\bm{x} \in B \mid \bm{x}[i] = \bm{l}[i], \bm{x}[j] = \bm{l}[j], \bm{x}[k] = \bm{l}[k]\} =: \base{ijk}{abc}
\]
holds, then $\base{i}{a} \vee \base{j}{b} = \base{j}{b} \vee \base{k}{c} = \base{k}{c} \vee \base{i}{a} = \bm{e}^{ijk}_{abc}$ by symmetry.
The set in the right-hand side is nonempty owing to $\bm{l}$.
By (ii) and (CT2), the $i$-th, $j$-th, and $k$-th components of $\bm{l}$ equal that of $\base{i}{a} \vee \base{j}{b}$.
Hence $\base{i}{a} \vee \base{j}{b}$ is in the set in the right-hand side of the equality.
In particular, $\base{i}{a} \vee \base{j}{b} \geq \base{ijk}{abc}$.
Conversely $\base{i}{a}$ and $\base{j}{b}$ is less than or equal to $\base{ijk}{abc}$ by LCP.
Therefore $\base{i}{a} \vee \base{j}{b} \leq \base{ijk}{abc}$.
We have thus proved that  $\base{i}{a} \vee \base{j}{b} = \base{ijk}{abc}$.
\end{proof}
We can enumerate all collinear triples in $O(n^3|\Lirr|^3)$-time by this proposition.
The condition (ii) can be decided in $O(1)$-time.
Thus we obtain all collinear triples in $O(n^3|\Lirr|^3)$-time by checking condition (ii) for all triples.
This completes the proof of Theorem \ref{thm:calcPPIP}.

\section{Implicational system for modular semilattice}
\label{sec:4}
Any semilattice $L$ is viewed 
as a $\cap$-closed family on $L^{\rm  ir}$, and 
is represented as an implicational system (or Horn formula) \cite{wild1994theory,wild2016the}.
In this section, we study implicational systems for modular semilattices.
In Section \ref{sec:4.1}, we determine optimal implicational bases for modular semilattices.
In Section \ref{sec:6.1}, we present 
a polynomial time recognition algorithm deciding whether a $\cap$-closed family given by implications is a modular semilattice.
Throughout this section,  modular semilattices are assumed to be finite.

\subsection{Optimal implicational base}
\label{sec:4.1}
We start with introducing basic terminologies in implicational systems,
which are natural adaptations of those in \cite{wild2016the} for our setting.
Fix a finite set $E$. 
A subset $\mathcal{F} \subseteq \power{E}$ is called a \textit{$\cap$-closed family} if $F_1 \cap F_2 \in \mathcal{F}$ for all $F_1, F_2 \in \mathcal{F}$.
The members of $\mathcal{F}$ is said to be \textit{closed}.
A $\cap$-closed family is naturally obtained from implications.
A pair of subsets $(A,B) \in \power{E} \times \power{E}$, written as $A \rightarrow B$, is called an \textit{implication}. 
Here $A$ is called the \textit{premise} and $B$ the \textit{conclusion}.
An implication is said to be \textit{proper} if its conclusion is nonempty. 
Let $\Sigma$ be a collection of implications.
We define a $\cap$-closed set $\mathcal{F}(\Sigma) \subseteq \power{E}$ as follows: 
$X \in \mathcal{F}(\Sigma)$ if and only if $A \subseteq X$ implies $B \subseteq X$ for all proper implications $A \rightarrow B$ in $\Sigma$, and $A \not \subseteq X$ for all improper implications $A \rightarrow \emptyset$.

A collection $\Sigma$ of implications is called an \textit{implicational base} of a $\cap$-closed family $\mathcal{F}$ if $\mathcal{F} = \mathcal{F}(\Sigma)$.
The \textit{size} of an implicational base $\Sigma$ is defined by
\[
s(\Sigma) := \sum_{(A \rightarrow B ) \in \Sigma} (|A| + |B|).
\]
An implicational base is said to be \textit{optimal} if its size is minimum among all implicational bases.

Modular semilattice $L$ can be viewed as a $\cap$-closed family.
In the previous section, we proved that $L$ is isomorphic to a $\cap$-closed family on \LIrr equipped with inclusion order $\subseteq$, that is, $\Csub{\PPIP{L}}$ in Theorem \ref{theo:BirkhoffSemiModular}.
A subset $X \subseteq \Lirr$ is said to be \textit{inconsistent} if there is no $F \in \Csub{\PPIP{L}}$ such that $X \subseteq F$.

Our aim is to give an optimal implicational base for modular semilattice $L$, viewed as a $\cap$-closed family $\Csub{\PPIP{L}}$.
Let $L$ be a modular semilattice.
An element $l' \in L$ is called a \textit{lower cover} of $l \in L$ if $l' < l$ and there is no element $l'' \in L$ such that $l' < l'' < l$. 
The relation $l' \prec l$ means that $l'$ is a lower cover of $l$.
Every \joinIrr element $q$ has the unique lower cover $\underline{q}$.
If $\underline{q}$ is not the minimum element, then $q$ is said to be \textit{nonatomic}.
For every nonatomic element $q$, its unique lower cover $\underline{q}$ is decomposed by \joinIrr elements $\{p_i\}$ as $\underline{q} = p_1 \vee p_2 \vee \dots \vee p_n$.
The subset $\{p_i\}$ is called an \textit{irreducible decomposition} of $\underline{q}$ if no proper subsequence $\{p_{i_k}\}$ decomposes $\underline{q}$, i.e., satisfies $\underline{q} = p_{i_1} \vee p_{i_2} \vee \dots \vee p_{i_m}$.
For a nonatomic \joinIrr element $q$, let $B_q = \{p_1, p_2, \dots, p_m \}$ denote an irreducible decomposition of  $\underline{q}$. 
An element $l \in L$ is called an \textit{$\MnElem$-element} ($n \geq 3$) if there are $y, x_0, x_1, \dots, x_{n-1} \in L$ with $y \prec x_i \prec l $ such that $x_i \wedge x_j = y$ and $x_i \vee x_j = l$ for all distinct $i,j$ in $\{0,1,\dots, n-1\}$. 
We call $y$ the bottom and $x_i$ an intermediate element.
A function $\phi \colon L \rightarrow \power{L^{\mathrm{ir}}}$ is defined by $\phi(l) = \{ p \in \Lirr \mid p \leq l\}$. 

Wild \cite{wild2000optimal} characterized an optimal implicational base for a modular lattice.
\begin{thm}[{\cite[PROPOSITION 5]{wild2000optimal}}]
\label{thm:optimalModular}
Let $L$ be a modular lattice. 
An optimal implicational base for $\Csub{\PPIP{L}}$ consists of the following implications:
\begin{itemize}
  \item $\{q\} \rightarrow B^q$ for every nonatomic $q \in \Lirr$;
  \item $\{ p_i^x, q_j^x \} \rightarrow \{ r_{j+ 1\mod n}^x \}$ for all $0 \leq i < j \leq n - 1$ and $\MnElem$-elements $x \in L$ with the bottom $y$ and intermediate elements $x_0, x_1, \dots, x_{n-1}$,
    where $p_i^x \in \phi(x_i) \setminus \phi(y) $, $q_j^x \in \phi(x_j) \setminus \phi(y)$, and $r^x_{j+ 1\mod n} \in \phi(x_{j+ 1\mod n}) \setminus \phi(y)$.
\end{itemize}
\end{thm}
We generalize this result for a modular semilattice. A pair $(p, q) \in \Lirr \times \Lirr$ is called a \textit{minimal inconsistent pair} if the following conditions are satisfied: $p \vee q$ does not exist; if $p' \leq p$, $q' \leq q$, and $p' \vee q'$ does not exist, then $p = p'$ and $q = q'$ for any $p', q' \in \Lirr$.
\begin{thm}
\label{thm:pesudoclosedSetForSemiModular}
Let $L$ be a modular semilattice. 
An optimal implicational base for $\Csub{\PPIP{L}}$ consists of the following implications:
\begin{itemize}
  \item $\{q\} \rightarrow B^q$ for every nonatomic $q \in \Lirr$;
  \item $\{ p_i^x, q_j^x \} \rightarrow \{ r_{j+ 1\mod n}^x \}$ for all $0 \leq i < j \leq n - 1$ and $\MnElem$-elements $x \in L$ with the bottom $y$ and intermediate elements $x_0, x_1, \dots, x_{n-1}$,
    where $p_i^x \in \phi(x_i) \setminus \phi(y) $, $q_j^x \in \phi(x_j) \setminus \phi(y)$, and $r^x_{j+ 1\mod n} \in \phi(x_{j+ 1\mod n}) \setminus \phi(y)$;
  \item $\{p,q\} \rightarrow \emptyset$ for every minimal inconsistent pair $(p,q) \in \Lirr \times \Lirr$.
\end{itemize}
\end{thm}
The compact representation by implicational bases is efficient when the modular semilattice $L$ contains a large \textit{diamond}.
A \textit{diamond} is a modular lattice whose height is two and whose maximum element is an $\MnElem$-element.
To represent a diamond by a PPIP, we need $O(n^3)$ collinear triples.
However optimal implicational base for it contains $O(n^2)$ implications.
\begin{ex}
An optimal implicational base for the modular semilattice in Figure \ref{fig:PPIPrep} (a) consists of the following implications:
$4 \rightarrow 5$;
$5 \rightarrow 1,3$;
$6 \rightarrow 1,2$;
$7 \rightarrow 2,3$;
$5,6 \rightarrow 7$;
$6,7 \rightarrow 5$;
$7,5 \rightarrow 6$;
$1,8 \rightarrow \emptyset$;
$2,4 \rightarrow \emptyset$.
\end{ex}
By using Theorem \ref{thm:pesudoclosedSetForSemiModular}, we can convert in polynomal time an implicational
base $\Sigma$ of a modular semilattice into optimal one, provided $E$ can be identified with ${\mathcal{F}^{\mathrm{ir}}(\Sigma)}$.
A family $\Sigma$ of implications is said to be {\it simple} if the
map $e \mapsto \bigcap \{X \in \mathcal{F}(\Sigma) | e \in X \}$ is a bijection
between $E$ and $\mathcal{F}^{\mathrm{ir}}(\Sigma)$. 

\begin{thm}
\label{thm:polynomialConvertion}
Any simple family $\Sigma$ of implications such that
$\mathcal{F}(\Sigma)$ is a modular semilattice, 
an optimal implicational base of $\mathcal{F}(\Sigma)$ is obtained in polynomial time.
\end{thm}
This is also a generalization of the result of Wild~\cite{wild2000optimal} for modular lattices.
The rest of this section is devoted to proving Theorem \ref{thm:pesudoclosedSetForSemiModular}, whereas
the proof of Theorem \ref{thm:polynomialConvertion} is given in the last of Section \ref{sec:6.1}.
We prove Theorem \ref{thm:pesudoclosedSetForSemiModular} by combining three previous results.
One is Arias and Balc\'{a}zar's reduction \cite{arias2009canonical} of a $\cap$-closed family to a \textit{closure system}, which is 
a $\cap$-closed family $\mathcal{F} \subseteq \power{E}$ with $E \in \mathcal{F}$.
For closure systems, several useful results are known.
By combining Arias and Balc\'{a}zar's reduction, they are generalized for $\cap$-closed systems.
The second is the characterization of optimal implicational bases for closure systems by \textit{pseudoclosed sets} \cite{wild2016the}.
The third is Wild's characterization \cite{wild2000optimal} of peseudoclosed sets of modular lattices.

We start with some definitions.
The closure operator $c$ of $\cap$-closed family $\mathcal{F}$ is defined by 
$c(X) := \bigcap \{F \in \mathcal{F} \mid X \subseteq F\}$.
If the set in the right-hand side is empty, $c(X)$ is undefined.
The closure operator gives the minimum closed set containing its operand.
Our closure operator is indeed a closure operator in a usual sense if $\mathcal{F}$ is a closure system.
It is easy to see that $c$ satisfies the following property: (monotonicity) $A \subseteq B \Rightarrow c(A) \subseteq c(B)$; (extensionality) $A \subseteq c(A)$; (idempotency) $c(c(A)) = c(A)$; where we assume the existence of $c(A)$ and $c(B)$.
The restriction of $\mathcal{F}$ to its closed set $A$, written as $\mathcal{F} \restriction_A$, is the closure system on $A$ defined by $\{F \in \mathcal{F} \mid F \subseteq A\}$.

Next we explain Arias and Balc\'{a}zar's reduction \cite{arias2009canonical}.
Let $F$ be a $\cap$-closed family.
For a new element $\perp$, let $E' := E \cup \{\perp\}$.
The \textit{reduced closure system} $\mathcal{F'} \subseteq \power{E'}$ of $\mathcal{F}$ is defined by $\mathcal{F'} := \mathcal{F} \cup \{E'\}$.
The \textit{reduced implicational base} $\Sigma'$ of $\Sigma$ is defined by \begin{align*}
\Sigma' = &\{A \rightarrow B \mid (A \rightarrow B) \in \Sigma , \;\;B \neq \emptyset \}\\
& \cup \{A \rightarrow \{\perp\} \mid (A \rightarrow \emptyset) \in \Sigma \}\\
& \cup \{ \{\perp\} \rightarrow E' \}.
\end{align*}
Notice that $\Sigma'$ is indeed an implicational base of $\mathcal{F'}$.
The operation $(\cdot)^{'}$ is an injection, and hence we can recover $\mathcal{F}$ and $\Sigma$ from $\mathcal{F'}$ and $\Sigma'$ of above forms, respectively.

We next characterize optimal implicational bases using \textit{pseudoclosed sets}.
Let $\mathcal{F}$ be a $\cap$-closed set and $c$ its closure operator.
A subset $X$ is said to be \textit{quasiclosed} if
\[
  \begin{cases}
    Y \subseteq X  \text{ and } c(Y) \neq c(X) \Rightarrow c(Y) \subseteq X & (\text{if } c(X) \text{ exists}), \\
    Y \subseteq X \text{ and } c(Y) \text{ exists} \Rightarrow c(Y) \subseteq X & (\text{if } c(X) \text{ does not exist}).
  \end{cases}
\]
It is clear that the family of quasiclosed sets forms a $\cap$-closed family.
We denote its closure operator by $c^{\bullet}$.
A \textit{properly quasiclosed} set is a quasiclosed but not closed set.
For any $X,Y \in \mathcal{F}$, the image of $X$ under $c$ is said to be \textit{equivalent} to that of $Y$ if and only if
\[
  \begin{cases}
   c(X) = c(Y) & (\text{if } c(X) \text{ exists}), \\
   c(Y) \text{ does not exist}  & (\text{if } c(X) \text{ does not exist}).
  \end{cases}
\]
A properly quasiclosed set $P$ is said to be \textit{psudoclosed} if $P$ is minimal among the properly quasiclosed sets whose images under $c$ is equivalent to that of $P$.
Our definition of pseudoclosed sets is a generalization of the standard one \cite{wild2016the}.
However these two definitions are closely related.
For closure systems, our definition of pseudoclosed sets coincides with the standard one.
Furthermore, for any $\cap$-closed system $\mathcal{F}$, a subset $P \in \power{E}$ is pseudoclosed in our sense if and only if $P$ is pseudoclosed of $\mathcal{F'}$ in the standard sense.

Optimal implicational bases are characterized as follows:
\begin{thm}[Essentially \cite{wild2016the}]
\label{thm:GDbaseSemi}
Let $\mathcal{F} \subseteq \power{E}$ be a $\cap$-closed family.
For every pseudoclosed set $P$, every implicational base $\Sigma$ of $\mathcal{F}$ contains an implication whose premise $A$ satisfies $A \subseteq P$ and $c^{\bullet}(A) = P$.
\end{thm}
\begin{proof}
We remarked above that $\Sigma'$ is an implicational base of closure system $\mathcal{F'}$.
It is known that this theorem holds for closure systems \cite{wild2016the}.
Since $P \in \power{E}$ is psudoclosed in our sense if and only if so is $P$ in $\mathcal{F'}$, 
there is an implication of the above form in $\Sigma'$.
Hence we have this theorem by recovering $\Sigma$ from $\Sigma'$ as mentioned above.
\end{proof}
By this theorem, we have a strategy to find optimal implicational bases:
find all pseudoclosed sets; 
then, for any pseudoclosed set $P$, minimize the cardinality of the premise and the conclusion of corresponding implications $A \rightarrow B$.

We next characterize pseudoclosed sets of $\Csub{\PPIP{L}}$ for modular semilattice $L$.
Wild \cite{wild2000optimal} characterized in the cases of modular lattices.
We generalize his result for modular semilattices.
Recall that $\phi \colon L \rightarrow \power{L^{\mathrm{ir}}}$ was defined by $\phi(l) := \{ p \in \Lirr \mid p \leq l\}$.
\begin{thm}[{\cite[PROPOSITION 4]{wild2000optimal}}]
\label{theo:pseudoclosedModular}
Let $L$ be a modular lattice and $c$ the closure operator of $\Csub{\PPIP{L}}$.
The family of pseudoclosed sets of $\Csub{\PPIP{L}}$ consists of the following subsets:
\begin{itemize}
  \item $\{q\}$ for every nonatomic $q \in L^{\mathrm{ir}}$;
  \item $\phi(x_i) \cup \phi(x_j)$ for every $0 \leq i < j \leq n-1$ and $\MnElem$-element $x \in L$ with intermediate elements $x_0, x_1, \dots, x_{n-1}$;
\end{itemize}
\end{thm}
This result is naturally generalized as follows:
\begin{thm}
\label{theo:pseudoclosedSemi}
Let $L$ be a modular semilattice and $c$ the closure operator of $\Csub{\PPIP{L}}$.
The family of pseudoclosed sets of $\Csub{\PPIP{L}}$ consists of the following subsets:
\begin{itemize}
  \item $\{q\}$ for every nonatomic $q \in L^{\mathrm{ir}}$;
  \item $\phi(x_i) \cup \phi(x_j)$ for every $0 \leq i < j \leq n-1$ and $\MnElem$-element $x \in L$ with intermediate elements $x_0, x_1, \dots, x_{n-1}$;
  \item $c^{\bullet}(\phi(x) \cup \phi(y))$ for every minimal inconsistent pair $x,y$ in $\PPIP{L}$.
\end{itemize}
\end{thm}
\begin{proof}
First we characterize consistent pseudoclosed sets. 
For convenience, we denote $\Csub{\PPIP{L}}$ by $\mathcal{F}$.
Let $S$ be a consistent pseudoclosed set of $\mathcal{F}$.
We prove that $S$ is of the first or second form in the statement by restricting $\mathcal{F}$ to $c(S)$.
The existence of $c(S)$ follows from the consistency of $S$ and Lemma \ref{lem:inductiveConstruction}.
Then the closure system $\mathcal{F} \restriction_{c(S)}$ is a modular lattice.
We can easily check that $F \in \mathcal{F} \restriction_{c(S)}$ is a pseudoclosed set of $\mathcal{F} \restriction_{c(S)}$ if and only if $F$ is a pseudoclosed set of $\mathcal{F}$.  
Hence $S$ is of the first or second form in the statement by Theorem \ref{theo:pseudoclosedModular}.
Conversely, let $S$ be the subset of the first or second form in the statement.
We can show that $S$ is pseudoclosed in a similar way. 
In the case where $S$ is of the first form, restrict $\mathcal{F}$ to $\phi(p)$.
Otherwise, restrict it to $\phi(x)$.

Next we characterize inconsistent pseudoclosed sets.
Let $S$ be an inconsistent pesudoclosed set of $\mathcal{F}$.
We prove that $S$ is of the third form in the statement.
Let $x, y$ be a minimal inconsistent pair in $S$.
Since $S$ is quasiclosed, it includes $c(\{x\})$ and $c(\{y\})$, that is, $\phi(x)$ and $\phi(y)$.
By the monotonicity and idempotency of $c^{\bullet}$, we have $c^{\bullet}(\phi(x) \cup \phi(y)) \subseteq c^{\bullet}(S) = S$.
Notice that $c^{\bullet}(\phi(x) \cup \phi(y))$ is properly quasiclosed and $c\left( c^{\bullet}(\phi(x) \cup \phi(y)) \right)$ does not exist.
By the minimality of $S$, the equality $S = c^{\bullet}(\phi(x) \cup \phi(y))$ holds.
We have thus shown that $S$ is of the third form in the statement.
Conversely, let $S = c^{\bullet}(\phi(x) \cup \phi(y))$ be a subset of the third form in the statement.
Since $X^{\bullet}$ is the minimum quasiclosed set including $X$, the subset $S$ is quasiclosed.
Furthermore, $S$ is properly quasiclosed since $c(S)$ does not exist.
We finally show that $S$ is pseudoclosed.
Suppose that $S' \subseteq S$ is quasiclosed and $c(S')$ does not exist.
The minimality of $x \smile y$ implies that $x$ and $y$ are in $S'$ by (IC2).
We can see that $\phi(x) \cup \phi(y) \subseteq S'$ and that $S = c^{\bullet}(\phi(x) \cup \phi(y)) \subseteq S'$ by the same argument as above.
This implies the minimality of $S$.
Thus $S$ is pseudoclosed.
\end{proof}
\begin{proof}[Proof of Theorem \ref{thm:pesudoclosedSetForSemiModular}]
We first prove that the set of implications in the statement is indeed an implicational base of $\mathcal{F}$. 
For convenience, we denote $\Csub{\PPIP{L}}$ by $\mathcal{F}$, and the collection of implications in the statement by $\Sigma$.
Let $c$ be the closure operator of $\mathcal{F}$ and $c'$ that of $\mathcal{F}(\Sigma)$.
We show by case analysis for the consistency that $c$ equals $c'$.
Let $S$ be an inconsistent subset of $\Lirr$.
Then $c'(S)$ does not exist since $\Sigma$ contains an implication of the third form.
Thus $c = c'$ for every inconsistent subset.
Let $S$ be a consistent subset of $\Lirr$.
Since every $F \in \mathcal{F}$ satisfies all implications in the statement, $\mathcal{F} \subseteq \mathcal{F}(\Sigma)$.
Therefore $c'(S) \subseteq c(S)$.
We prove $c(S) = c'(S)$ by the restriction of $\mathcal{F}$ to $c(S)$.
Wild \cite{wild2000optimal} showed that restriction of implications in the statement to $c(S)$ is indeed an implicational base of $\mathcal{F} \restriction_{c(S)}$
Therefore $c(F) = c'(F)$ for any $F \subseteq S$.
In particular, $c(S) = c'(S)$.
We have thus proved that $c$ is equal to $c'$.

We next prove the optimality of the implicational base in the statement.
Every implication above corresponds to a pseudoclosed set characterized in Theorem \ref{theo:pseudoclosedSemi}.
Hence, by Theorem \ref{thm:GDbaseSemi}, it suffices to show that the cardinality of the premise and conclusion is minimum for every implication above.

Let us show that the cardinalities of premises are minimum.
Note that no premises are empty since the minimum element of $\mathcal{F}$ is $\emptyset$.
Hence the cardinality of the premise is minimum for every implications of the first form.
Every implication of the second form corresponds to the pesudoclosed set $\phi(x_i) \cup \phi(x_j)$, 
where $x_i$ and $x_j$ are the intermediate elements of some $\MnElem$-element $x$.
Any implication of singleton premise $\{p\}$ cannot satisfy $c^{\bullet}(\{p\}) = \phi(x_i) \cup \phi(x_j)$ since $c^{\bullet}(\{p\}) = \{p\}$.
Thus the premises of implications of the second form have minimum cardinalities.
We can show that the premises of implications of the third form have minimum cardinalities in a similar fashion.

We finally prove that the cardinalities of conclusions are minimum.
The conclusions of implications of the first form have minimum cardinalities since every irreducible decomposition has the minimum length \cite[Proposition 2.23]{aigner2012combinatorial}.
That of the second and the third form clearly have minimum cardinalities.
\end{proof}
\subsection{Recognition algorithm}
\label{sec:6.1}
Here we present a recognition algorithm deciding whether a $\cap$-closed family given by implications is a modular semilattice.
Our algorithm is a natural extension of Herrmann and Wild's algorithm \cite{herrmann1996polynomial} deciding whether a closure system given by implications is a modular lattice.
\begin{thm}
\label{thm:RecognizeSemiModular}
Let $\Sigma$ be a family of implications on $E$.
We can decide whether or not $\mathcal{F}(\Sigma)$ is a modular semilattice in polynomial time.
\end{thm}
The rest of this section is devoted to the proof of this theorem.
Let $\Sigma$ be a family of implications on $E$.
An operator $i_{\Sigma}$ on $\power{E}$ is defined by $i_{\Sigma}(P) := P \cup \bigcup \{B \mid (A \rightarrow B) \in \Sigma, A \subseteq P\}$.
Let $i^n_{\Sigma}$ denote the $n$ times composition of $i_{\Sigma}$. 

\begin{lem}
\label{lem:calcClosure}
Let $\Sigma$ be a family of implications on $E$ and $c$ the closure operator of $\mathcal{F}(\Sigma)$.
For any $X \subseteq E$, we can compute $c(X)$ in $O(s(\Sigma))$-time.
\end{lem}
\begin{proof}
Let $c'$ be the closure operator of the reduced closure system $\mathcal{F}(\Sigma')$.
Notice that the size of $\Sigma'$ is $O(s(\Sigma))$.
It is known that we can compute $c(X)$ if $\Sigma$ includes no improper implications \cite[p. 65]{maier1983theory}.
Hence we can compute $c'(X)$ in $O(s(\Sigma))$-time.
Then $c(X)$ equals $c'(X)$ if $c'(X) \subseteq E$, and $c(X)$ does not exist if $\perp \in c'(X)$.
\end{proof}
\begin{lem}
\label{lem:implication2PPIP}
Let $\Sigma$ be a family of implications.
We can compute $(\mathcal{F}^{\mathrm{ir}}(\Sigma), \leq, \smile, C)$ in $O(s(\Sigma)|E|^3)$-time, where $\smile$ and $C$ are the induced inconsistency and collinearity relation of the semilattice $\mathcal{F}(\Sigma)$, respectively.
\end{lem}
\begin{proof}
We first enumerate the \joinIrr elements of $\mathcal{F}(\Sigma)$.
Since $c(X) = \bigvee_{x \in X} c(\{x\})$, \joinIrr elements are of the form $c(\{x\})$ for some $x \in X$.
Thus it suffices to compute $c(\{x\})$ for each $x \in X$ and to check whether they are indeed \joinIrrLast{} elements.
The partial order, inconsistency relation, and collinearity relation computed by a brute force search on the enumerated $\vee$-irreducible elements.
By Lemma \ref{lem:calcClosure}, it is clear that this whole process can be done in $O(s(\Sigma)|E|^3)$-time.
\end{proof}
\begin{lem}
\label{lem:closureEquation}
Let $\Sigma$ be a family of implications on $E$ and $c$ be the closure operator of $\mathcal{F}(\Sigma)$.

\begin{itemize}
\item[{\rm (1)}] For any consistent $X \subseteq E$, its closure $c(X)$ is given by $c(X) = X \cup i_{\Sigma}^{1}(X) \cup i_{\Sigma}^{2}(X) \cup \cdots$.
\item[{\rm (2)}] For any inconsistent $X \subseteq E$, there exist $k \in \mathbb{N}$ and an improper implication $(A \rightarrow \emptyset) \in \Sigma$ such that $A \subseteq i_{\Sigma}^{k}(X)$. 
\end{itemize}
\end{lem}
\begin{proof}
It is known that $c(X) = X \cup i_{\Sigma}^{1}(X) \cup i_{\Sigma}^{2}(X) \cup \cdots$ if $\Sigma$ contains no improper implications.
Thus $c'(X) =  X \cup i_{\Sigma'}^{1}(X) \cup i_{\Sigma'}^{2}(X) \cup \cdots$, where $c'$ is the closure operator of $\mathcal{F}(\Sigma')$.
In the case where $X$ is consistent, $c'(X) \subseteq E$.
Hence all $i^{n}_{\Sigma'}(X)$ are contained in $E$ and equals $i^{n}_{\Sigma}(X)$.
Thus we have $c(X) = X \cup i_{\Sigma}^{1}(X) \cup i_{\Sigma}^{2}(X) \cup \cdots$.
This completes the proof of (1).
In the case where $X$ is inconsistent, $\perp \in E$.
Let $k$ be the maximum positive integer such that $\perp \not \in i^k_{\Sigma'}(X)$.
It follows from the definition of $k$ and $i_{\Sigma}$ that $i_{\Sigma'}^{k}(X)$ contains the premise of an improper implication in $\Sigma$.
Since $i_{\Sigma'}^{k}(X) \subseteq E$, we have $i_{\Sigma}^{k}(X) = i_{\Sigma'}^{k}(X)$.
Thus there exist $k \in \mathbb{N}$ and an improper implication $(A \rightarrow \emptyset) \in \Sigma$ such that $A \subseteq i_{\Sigma}^{k}(X)$.
This completes the proof of (2).
\end{proof}
For simplicity, we assume that $\Sigma$ satisfies the following condition:
For any $e \in E$, the closure $c(\{e\})$ exist.
This assumption loses no generality.
If $c(\{e\})$ does not exist, remove $e$ from $E$.
A pair $e_1, e_2 \in E$ is called an \textit{inconsistent pair} if $c(\{e_1, e_2\})$ does not exist.
\begin{lem}
\label{lem:InconsistentImplication}
Let $\Sigma$ be a family of implications on $E$ and $c$ the closure operator of $\mathcal{F}(\Sigma)$.
Then the following conditions are equivalent:
\begin{itemize}
  \item[{\rm(1)}] For any $X,Y,Z \subseteq \mathcal{F}(\Sigma)$, the join $X \vee Y \vee Z$ exists if $X \vee Y$, $Y \vee Z$, and $Z \vee X$ exist. 
  \item[{\rm (2)}] For any $X \subseteq E$, the closure $c(X)$ does not exist if and only if $X$ has an inconsistent pair.
  \item[{\rm (3)}] $\Sigma$ satisfies the following:
  \begin{itemize}
    \item[{\rm (i)}] For any improper implication $(A \rightarrow \emptyset) \in \Sigma$, the premise $A$ contains an inconsistent pair.
    \item[{\rm (ii)}] For any $(A \rightarrow B), (A' \rightarrow B') \in \Sigma$, if the set $A \cup A' \cup B \cup B'$ contains an inconsistent pair, then  so does $A \cup A'$.
    \item[{\rm (iii)}] For any $(A \rightarrow B) \in \Sigma$ and $e \in E$,  if the set $A \cup B \cup \{e\}$ contains an inconsistent pair, then so does $A \cup \{e\}$.
  \end{itemize}
\end{itemize}
\end{lem}
\begin{proof}
(1) $\Rightarrow$ (2):
Obviously, $c(X)$ does not exist if $X$ has an inconsistent pair.
Conversely, suppose that $c(X)$ does not exist.
We prove that $X$ contains an inconsistent pair by induction on $|X|$. The case $|X| = 2$ is trivial.
In the case $|X| > 2$, let $x_1, x_2, x_3$ be distinct elements in $X$.
Let $X_1 = X \setminus \{x_2, x_3\}$, $X_2 = X \setminus \{x_3, x_1\}$, and $X_3 = X \setminus \{x_1, x_2\}$.
By (1) and the fact that $c(X) = c(X_1) \vee c(X_2) \vee c(X_3)$ does not exist, one of the $c(X_1) \vee c(X_2)$, $c(X_2) \vee c(X_3)$, and $c(X_3) \vee c(X_1)$ does not exist.
In particular, one of the $c(X_1 \cup X_2)$, $c(X_2 \cup X_3)$, and $c(X_3 \cup X_1)$ does not exist.
By induction, $X_1 \cup X_2$, $X_2 \cup X_3$, or $X_3 \cup X_1$ contains an inconsistent pair.
Thus $X$ contains an inconsistent pair.

(2) $\Rightarrow$ (1): 
We prove the contraposition of (1), that is, for any $X,Y,Z \in \mathcal{F}(\Sigma)$, one of the $X \vee Y$, $Y \vee Z$, and $Z \vee X$ does not exist if $X \vee Y \vee Z$ does not exist.
Suppose that $X \vee Y \vee Z = c(X \cup Y \cup Z)$ does not exist.
By (2), the set $X \cup Y \cup Z$ contains an inconsistent pair $a,b$.
Then one of the $X \cup Y$, $Y \cup Z$, and $Z \cup X$ contains $a,b$.
By using (2) again, one of $X \vee Y$, $Y \vee Z$, and $Z \vee X$ does not exist.

(2) $\Rightarrow$ (3): 
The condition (i) follows from (2) since $c(A)$ does not exist.
For the condition (ii), suppose that $A \cup A' \cup B \cup B'$ contains an inconsistent pair.
By (2), the closure $c(A \cup A' \cup B \cup B')$ does not exist.
Since $c(A \cup A') = c(A \cup A' \cup B \cup B')$, the set $A \cup A'$ contains an inconsistent pair by (2).
We can prove the condition (iii) in a similar way.

(3) $\Rightarrow$ (2):
Obviously $c(X)$ does not exist if $X$ has an inconsistent pair.
Conversely, suppose that $c(X)$ does not exist.
By Lemma \ref{lem:closureEquation} (2), there is $k \in \mathbb{N}$ and $(A \rightarrow \emptyset) \in \Sigma$ such that 
$A \subseteq i^k_{\Sigma}(X)$.
We prove that $X$ has an inconsistent pair by induction on $k$.
In the case $k = 0$, there is an improper implication $(A \rightarrow \emptyset) \in \Sigma$ such that $A \subseteq X$.
By (i), the premise $A$ contains an inconsistent pair.
Hence $X$ contains an inconsistent pair.
In the case $k > 0$, the set $i_{\Sigma}(X)$ contains an inconsistent pair $a,b$ by induction.
By the definition of $i_{\Sigma}$, one of the following cases hold:
there are $A \rightarrow B$ and $A' \rightarrow B'$ in $\Sigma$ such that $A \subseteq X$, $ A' \subseteq X$, $a \in B$, and $b \in B'$;
$a \in X$ and there is $(A \rightarrow B) \in \Sigma$ such that $A \subseteq X$ and $b \in B$.
By (3), the set $A \cup A'$ or $A \cup \{a\}$ contains an inconsistent pair respectively.
Hence $X$ contains an inconsistent pair.
\end{proof}
\begin{proof}[Proof of Theorem \ref{thm:RecognizeSemiModular}]
We need to check that $\mathcal{F}(\Sigma)$ satisfies the following two conditions:
(JOIN) for any $X,Y,Z \subseteq \mathcal{F}(\Sigma)$, the join $X \vee Y \vee Z$ exist if $X \vee Y$, $Y \vee Z$, and $Z \vee X$ exist;
(MOD) Every principal ideal is a modular lattice. 
We have proved that (JOIN) is equivalent to the third condition of Lemma \ref{lem:InconsistentImplication}.
The third condition of Lemma \ref{lem:InconsistentImplication} can be checked in $O((|\Sigma|^2|E|^2+ |\Sigma||E|^3) s(\Sigma))$-time.
Suppose that $\mathcal{F}(\Sigma)$ satisfies (JOIN).
We later show that $\mathcal{F}(\Sigma)$ satisfies (MOD) if and only if  $(\mathcal{F}^{\mathrm{ir}}(\Sigma), \leq, \smile, C)$satisfies Regularity and weak Triangle axioms.
The latter condition can be decided in $O(s(\Sigma)|E|^3 + |E|^7 \log|E|)$-time by Lemma \ref{lem:implication2PPIP}, which completes the proof.

We prove that $\mathcal{F}(\Sigma)$ satisfies (MOD) if and only if  $(\mathcal{F}^{\mathrm{ir}}(\Sigma), \leq, \smile, C)$ satisfies Regularity and weak Triangle axioms.
It follows from Theorem \ref{theo:BirkhoffSemiModular} that $(\mathcal{F}^{\mathrm{ir}}(\Sigma), \leq, \smile, C)$ satisfies Regularity and weak Triangle axioms if $\mathcal{F}(\Sigma)$ satisfies (MOD).
Conversely, suppose that  $(\mathcal{F}^{\mathrm{ir}}(\Sigma), \leq, \smile, C)$ satisfies Regularity and weak Triangle axioms.
Let $X \in \mathcal{F}(\Sigma)$.
Notice that weak Triangle axiom is equivalent to Triangle axiom on $X$, where we regard $X$ is the substructure of $(\mathcal{F}^{\mathrm{ir}}(\Sigma), \leq, \smile, C)$.
Notice that $X = (I_X^{\mathrm{ir}}, \leq, \smile, C)$ as a substructure of  $(\mathcal{F}^{\mathrm{ir}}(\Sigma), \leq, \smile, C)$.
Herrmann and Wild \cite{herrmann1996polynomial} proved that, for any lattice $L$, the lattice $L$ is modular if and only if $(\Lirr, \leq, \smile, C)$ satisfies Regularity and Triangle axioms.
Hence $I_X$ is a modular lattice.
We have thus proved that $\mathcal{F}(\Sigma)$ satisfies (MOD). 
\end{proof}
\begin{proof}[Proof of Theorem \ref{thm:polynomialConvertion}]
Since there is a natural bijection between ${\mathcal{F}^{\mathrm{ir}}(\Sigma)}$ and $E$, it suffices to compute an optimal implicational base of $\Csub{\PPIP{\mathcal{F}(\Sigma)}}$, which is characterized in Theorem \ref{thm:pesudoclosedSetForSemiModular}.
Notice that this bijection induces an isomorphism of $\cap$-closed families between $\mathcal{F}(\Sigma)$ and $\Csub{\PPIP{\mathcal{F}(\Sigma)}}$.
Throughout this proof, we identify $E$ with ${\mathcal{F}^{\mathrm{ir}}(\Sigma)}$, and $\mathcal{F}(\Sigma)$ with $\Csub{\PPIP{\mathcal{F}(\Sigma)}}$.

We first compute implications of the first and third form in Theorem \ref{thm:pesudoclosedSetForSemiModular}.
The PPIP-representation $\PPIP{\mathcal{F}(\Sigma)}$ is obtained in polynomial time by Lemma \ref{lem:implication2PPIP}.
Thus we have the premises of implications of the first and third form.
For any \joinIrr element $q$, it is easy to compute the conclusion $B^q$ from the PPIP-representation.

We next compute implications of the second form.
It suffices 
to compute $ \phi(x_i)$ and $\phi(x_j)$ for every pseudoclosed set $Q = \phi(x_i) \cup \phi(x_j)$ of the second form in Theorem \ref{theo:pseudoclosedSemi}.
Let $c$ be the closure operator of $\Csub{\PPIP{\mathcal{F}(\Sigma)}}$.
All pseudoclosed sets are obtained via the \textit{GD-base}~\cite{Guigues1986, maier1983theory} of the reduced closure system.
See also~\cite[p. 386]{herrmann1996polynomial} for the computation of the GD-base from $\Sigma$.
Then we can pick up pseudoclosed sets of the second form 
since a pesudoclosed set $P$ is of the second form if and only if $c(P)$ exists and $P$ is not a \joinIrr element.
Let $Q = \phi(x_i) \cup \phi(x_j)$ be a pesudoclosed set of the second form.
We can compute $\phi(x_i)$ and $\phi(x_j)$ as follows.
Let $\mathcal{P} := \{X \in \power{E} \mid X \text{ is pseudoclosed and } c(X) = c(Q)\}$.
By the definition of $\mathcal{M}_n$-elements,
$\phi(x_i) \cap \phi(x_j) = \bigcap \mathcal{P}$.
Furthermore, the image of the map $\mathcal{P} \ni X \mapsto X \cap Q \in \power{E}$ is $\{\phi(x_i), \phi(x_j), \phi(x_i) \cap \phi(x_j), \phi(x_i) \cup \phi(x_j)\}$.
Therefore we can compute $\phi(x_i)$ and $\phi(x_j)$.
\end{proof}
\section{Application}
\label{sec:7}
In this section, we mention possible applications of our results.
\subsection{Minimizer set of submodular function}
\label{sec:7.1}
The motivation of this paper comes from
\textit{submodular functions} on modular semilattices \cite{Hirai0ext,hirai2016lconvexity}.
This class of functions generalizes submodular set functions as well
as other submodular-type functions,
such as $k$-submodular functions,
and appears from dualities in well-behaved multicommodity flow problems
and related label assignment problems; see also \cite{Hirai2017survery}.

A submodular function on modular semilattice $L$ is 
a function $f:L \to \mathbb{R} \cup \{+\infty\}$ satisfying
\[
f(x) + f(y) \geq f(x \wedge y) + 
\sum_{\theta \in {\cal E}(L)} C (\theta) f( \theta(x, y)) \quad (x,y \in L),
\]
where ${\cal E}(L)$ is a certain set of binary operators on $L$, 
and $C:{\cal E}(L) \to \RR_+$ 
is a probability distribution on ${\cal E}(L)$.
We do not give the detailed definition of ${\cal E}({\cal L})$ and $C$.
An important point here is that 
each operator $\theta$ in ${\cal E}({\cal L})$ is {\em $\vee$-like} 
in the sense that $\theta(x,y) = x \vee y$ holds for any $x,y$ having the join. Thus, if $x$ and $y$ are minimizers of a submodular function, 
then $x \wedge y$ and $\theta(x,y)$ with $C(\theta) > 0$ are minimizers. In addition, if $x$ and $y$ have the join, 
then $x \vee y = \theta(x,y)$ is also a minimizer. 
\begin{lem}
	The minimizer set of a submodular function 
	on a modular semilattice is $(\wedge, \vee)$-closed.
\end{lem}
Thus the minimizer set $B$ of a submodular function $f$ 
is represented as a PPIP.
Currently no polynomial time algorithm is known for 
minimizing  this class of submodular functions.
We here consider a typical case where $f$ is defined on 
the $n$-product $L^n$ of a (small) modular semilattice $L$, and is of the form   
\begin{equation}
\label{eq:zivny}
f(\bm{x}) = \sum_{i = 1}^N f_i(\bm{x}[i_1], \bm{x}[i_2], \dots, \bm{x}[i_m]),
\end{equation} 
where each $f_i$ is a submodular function on $L^m$ and 
$m$ is a (small) constant.
In this case, 
$f$ is minimized in time polynomial of $n$, $N$, and $|L|^m$, 
which is a consequence of a general tractability criterion 
by Kolmogorov, Thapper, and \v{Z}ivn\'y \cite{kolmogorov2015the}
for minimizing a function of the form (\ref{eq:zivny}).
Then the membership oracle (MO) of $B$ 
can be obtained from a minimizing oracle of $f$ together with a variable-fixing procedure.
Also, by Theorem~\ref{theo:BirkhoffSemiModular} and \ref{thm:characterizeJoinIrr}, 
the minimizer set $B$ of $f$ is compactly written as 
a PPIP of $O(n|L|)$ elements.
Thus, 
by the PPIP-construction algorithm in Section \ref{sec:5}, 
we obtain the following.
\begin{thm}\label{thm:minimizer}
	Let $L$ be a modular semilattice, 
	and let $f$ be a function on $L^n$ of form $(\ref{eq:zivny})$ such that each $f_i$ is submodular on $L^m$.
	The PPIP-representation of the minimizer set of $f$ 
	is obtained in time polynomial in $n$, $N$, and $|L|^m$.
\end{thm}
This is an extension of a result of Hirai and Oki \cite{hirai2016lecture} for the case of $k$-submodular functions.

\subsection{Block-triangularization of partitioned matrix}
\label{sec:7.2}
The DM-decomposition of a matrix is 
obtained from a maximal chain of the minimizer set of a submodular function.
Here we apply our results 
to a generalization of the DM-decomposition 
considered by Ito, Iwata, and Murota \cite{ito1994blocktriangularizations}, 
in which a submodular function on a modular lattice plays a key role.

A \textit{partitioned matrix} of \textit{type} $(m_1, m_2, \dots, m_{\mu}; n_1, n_2, \dots, n_{\nu})$ is any matrix $A = (A_{\alpha \beta})$ having a block-matrix structure 
as \[
A = \left(
\begin{array}{cccc}
A_{11} & A_{12} & \ldots & A_{1\nu} \\
A_{21} & A_{22} & \ldots & A_{2\nu} \\
\vdots & \vdots & \ddots & \vdots \\
A_{\mu 1} & A_{\mu 2} & \ldots & A_{\mu \nu}
\end{array}
\right),
\]
where $A_{\alpha \beta}$ is an $m_{\alpha} \times n_{\beta}$ matrix over field ${\bf F}$ for $\alpha \in [\mu]$ and $\beta \in [\nu]$.
Let $m = \sum_{\alpha \in [\mu]} m_{\alpha}$ and $n = \sum_{\beta \in [\nu]} n_{\beta}$.
Ito, Iwata, and Murota \cite{ito1994blocktriangularizations} showed that partitioned matrix $A = (A_{\alpha \beta})$ admits a canonical block-triangular form 
\[
\left(
\begin{array}{cccc}
D_{1} & *  & \ldots & * \\
O & D_{2} & \ddots & \vdots \\
\vdots & \ddots & \ddots &  *\\
O & \ldots & O& D_{\kappa}
\end{array}
\right)
\]
under transformations of form $A \mapsto$
\[
P \left(
    \begin{array}{cccc}
      E_{1} & O & \ldots & O \\
      O & E_{2} & \ddots & \vdots \\
      \vdots & \ddots & \ddots & O \\
      O & \ldots & O& E_{\mu}
    \end{array}
  \right) 
  \left(
    \begin{array}{cccc}
      A_{11} & A_{12} & \ldots & A_{1\nu} \\
      A_{21} & A_{22} & \ldots & A_{2\nu} \\
      \vdots & \vdots & \ddots & \vdots \\
      A_{\mu 1} & A_{\mu 2} & \ldots & A_{\mu \nu}
    \end{array}
  \right)
  \left(
    \begin{array}{cccc}
      F_{1} & O & \ldots & O \\
      O & F_{2} & \ddots & \vdots \\
      \vdots & \ddots & \ddots & O \\
      O & \ldots & O& F_{\nu}
    \end{array}
  \right)
  Q,
\]
where each $*$ is an arbitrary element, $E_{\alpha}$ is a nonsingular $m_{\alpha} \times m_{\alpha}$ matrix, $F_{\beta}$ is a nonsingular $n_{\beta} \times n_{\beta}$ matrix, and $P$ and $Q$ are permutation matrices of sizes $m$ and $n$, respectively.

This block-triangularization is obtained from an optimization 
over subspaces of vector spaces.
A $\mu+\nu$ tuple $(X,Y) = (X_1,X_2,\ldots,X_\mu,Y_1,Y_2,\ldots,Y_\nu)$ of subspaces $X_{\alpha} \subseteq {\bf F}^{m_\alpha}$ for $\alpha \in [\mu]$ and 
$Y_{\beta} \subseteq {\bf F}^{n_\beta}$ for $\beta \in [\nu]$ 
is said to be {\em vanishing} if
\begin{equation}\label{eqn:vanishing}
A_{\alpha \beta}(X_{\alpha},Y_{\beta}) = \{0\} \quad (\alpha \in [\mu],\beta \in [\nu]),
\end{equation}
where $A_{\alpha \beta}$ is regarded as a bilinear form 
$(u,v) \mapsto u^{\top}A_{\alpha \beta} v$.
We simply call such $(X,Y)$ a {\em vanishing subspace}, 
where $X$ is a $\mu$-tuple of subspaces and $Y$ is a $\nu$-tuple of subspaces.

For a vanishing subspace $(X,Y)$, we obtain 
a transformation of the above form so that 
the transformed matrix has a zero submatrix of 
$\sum_{\alpha} \dim X_{\alpha}$ rows and $\sum_{\beta} \dim Y_{\beta}$.
This naturally leads us to the following problem, called  
the {\em maximum vanishing subspace problem} (MVSP).
The goal of MVSP is to
maximize 
\begin{equation}
\sum_{\alpha} \dim X_{\alpha} + \sum_{\beta} \dim Y_{\alpha}
\end{equation}
over all vanishing subspaces $(X,Y)$ for $A$.
MVSP is an algebraic generalization 
of the stable set problem of a bipartite graph, and was implicit in \cite{ito1994blocktriangularizations} and formally introduced by~\cite{hirai2016computing}.

From a maximal chain $(X^0,Y^0), (X^1,Y^1),\ldots, (X^{\kappa},Y^{\kappa})$ 
of maximum vanishing subspaces, 
where $X^i_{\alpha} \subseteq X^{i+1}_{\alpha}$
 and $Y^{i}_{\beta} \supseteq Y^{i+1}_{\beta}$, 
we obtain, via an appropriate change of bases, 
the most refined block-triangular form, which we call the {\em DM-decomposition} of $A$. See \cite{hirai2016computing} for detail.

Currently a polynomial time algorithm to obtain the DM-decomposition is known 
for very restricted classes of partitioned matrices \cite{hirai2016computing,murota1987ccf}.
Just recently, Hamada and Hirai \cite{hamada2017maximum} proved 
that MVSP can be solved in polynomial time.
However this result is not enough to obtain the DM-decomposition, 
since the DM-decomposition needs 
a maximal chain 
of maximum vanishing subspaces and their algorithm outputs one of (special) maximum vanishing spaces.

As remarked in~\cite{hirai2016computing},
MVSP is formulated as a submodular function minimization on 
the product of modular lattices. 
For $\alpha \in [\mu]$, let $L_{\alpha}$ denote 
the family of all subspaces of ${\bf F}^{m_{\alpha}}$, 
and for $\beta \in [\nu]$, let $M_{\beta}$ denote 
the family of all subspaces of ${\bf F}^{n_{\beta}}$.
Both $L_{\alpha}$ and $M_\beta$ are modular lattices with respect to  inclusion/reverse inclusion order, where $L_{\alpha}$ is considered in inclusion order and $M_{\beta}$ is considered in reverse inclusion order.
Define $S_{\alpha \beta}: L_{\alpha} \times M_{\beta} \to {\bf R} \cup \{\infty\}$ 
by $S_{\alpha}(X_{\alpha},Y_{\beta}) := 0$ if $A_{\alpha \beta}(X_{\alpha},Y_{\beta}) = \{0\}$ and $\infty$ otherwise.
Then MVSP is equivalent to 
\begin{eqnarray*}
{\rm Min}. && - \sum_{\alpha} \dim X_{\alpha} - \sum_{\beta} \dim Y_{\beta} 
+ \sum_{\alpha,\beta} S_{\alpha \beta}(X_{\alpha}, Y_\beta) \\
{\rm s.t.} && (X_1,X_2,\ldots,X_\mu,Y_1,Y_2,\ldots,Y_{\nu})  \\
 && \quad \quad \quad \in L_1 \times L_2 \times \cdots \times L_\mu \times M_1 \times M_2 \times \cdots \times M_\nu.  
\end{eqnarray*}
The objective function is of the form~(\ref{eq:zivny}) with $m=2$, 
and is submodular \cite{hirai2016computing}. 

If ${\bf F}$ is a finite field, then $L_{\alpha}$ and $M_{\beta}$
are all finite, and Theorem~\ref{thm:minimizer} is applicable to obtain
the PPIP-representation $\PPIP{B}$ of the minimizer set $B$ of this submodular function.
A maximal chain $\{\subseteqInd{k}\}$ (of consistent subspaces in $\PPIP{B}$) is obtained in a greedy manner:
Let $S^{0} := \emptyset$.
Choose any minimal element $p \in \PPIP{B} \setminus \subseteqInd{k}$, 
let 
\begin{equation*}
\label{eq:funi}
\subseteqInd{k+1} := \subseteqInd{k} \cup \{p\} \cup \{r \in \PPIP{B} \mid \exists q \in \subseteqInd{k} \text{ such that } C(p,q,r) \text{ holds}\}.
\end{equation*}
By Lemma \ref{lem:joinSemiModular}, 
$\subseteqInd{k+1}$ is indeed a consistent subspace, 
and has $\subseteqInd{k}$ as a lower cover.
\begin{thm}
	Let $A = (A_{\alpha \beta})$ be an $m \times n$ partitioned matrix 
	over a finite field ${\bf F}$.
	Then the DM-decomposition of $A$ can be obtained in time polynomial 
	in $m$, $n$, and the number of all vector subspaces 
	of ${\bf F}^{\gamma}$, where $\gamma$ is 
	the maximum number of rows and columns of a submatrix $A_{\alpha \beta}$.
\end{thm}
This is the first general framework to compute 
the DM-decomposition of partitioned matrix on a finite field.
\begin{ex}
\label{ex:maximalchain}
Let us compute the DM-decomposition of matrix $A$ on ${\bf F} := {\rm GF}(2)$ defined as
\[A  =: \left(
    \begin{array}{cc|cc|cc}
      1 & 1 & 0 & 1 & 0 & 1\\
      0 & 1 & 0 & 0 & 1 & 1\\ \hline
      0 & 0 & 0 & 0 & 1 & 1\\
      0 & 0 & 0 & 0 & 1 & 0\\ \hline
      1 & 1 & 1 & 1 & 0 & 1\\
      0 & 1 & 1 & 0 & 0 & 0
    \end{array}
  \right). \]
 The PPIP representation (or projective ordered space) of 
 the modular lattice of maximum vanishing spaces
 is illustrated in Figure \ref{fig:maximalChain}.
Let $\bm{e}_1 := (0,1)^\top$, $\bm{e}_2 := (1,1)^\top$, and $\bm{e}_3 := (1,0)^\top$.
Each \joinIrr element corresponds to the subspace defined by
\begin{enumerate}
  \item $(\{0\},  \mathrm{span}(\bm{e}_1) , \{0\} , {\bf F}^2 , {\bf F}^2 ,\mathrm{span}(\bm{e}_1)))$,
  \item $(\{0\} , \mathrm{span}(\bm{e}_2) , \{0\} ,  {\bf F}^2 , {\bf F}^2 , \mathrm{span}(\bm{e}_3))$,
  \item $(\{0\}, \mathrm{span}(\bm{e}_3), \{0\},  {\bf F}^2, {\bf F}^2, \mathrm{span}(\bm{e}_2))$,
  \item $(\mathrm{span}(\bm{e}_2), {\bf F}^2, \mathrm{span}(\bm{e}_2), \mathrm{span}(\bm{e}_1), \mathrm{span}(\bm{e}_3), \{0\})$,
  \item $(\mathrm{span}(\bm{e}_1), \mathrm{span}(\bm{e}_3), \{0\}, \mathrm{span}(\bm{e}_3), {\bf F}^2, \mathrm{span}(\bm{e}_2))$,
  \item $(\mathrm{span}(\bm{e}_1), \mathrm{span}(\bm{e}_3), \mathrm{span}(\bm{e}_1), \mathrm{span}(\bm{e}_3), \mathrm{span}(\bm{e}_1), \mathrm{span}(\bm{e}_2))$,
\end{enumerate}
in the numerical order, where $\mathrm{span}(v)$ is the subspace spanned by a vector $v$.
From this maximal chain, the DM-decomposition of $A$ is obtained as follows:
\[\left(
    \begin{array}{cccccc}
      1 & 1 & 0 & 0 & 1 & 0\\
      1 & 1 & 0 & 0 & 0 & 1\\
      0 & 0 & 1 & 1 & 0 & 0\\
      0 & 0 & 0 & 1 & 1 & 1\\
      0 & 0 & 0 & 0 & 1 & 0\\
      0 & 0 & 0 & 0 & 0 & 1
    \end{array}
  \right)
  =\left(
    \begin{array}{cccccc}
      1 & 1 & 0 & 0 & 0 & 0\\
      0 & 0 & 0 & 0 & 1 & 1\\
      0 & 0 & 0 & 0 & 0 & 1\\
      0 & 1 & 0 & 0 & 0 & 0\\
      0 & 0 & 0 & 1 & 0 & 0\\
      0 & 0 & 1 & 1 & 0 & 0
    \end{array}
  \right)
 \left(
    \begin{array}{cc|cc|cc}
      1 & 1 & 0 & 1 & 0 & 1\\
      0 & 1 & 0 & 0 & 1 & 1\\ \hline
      0 & 0 & 0 & 0 & 1 & 1\\
      0 & 0 & 0 & 0 & 1 & 0\\ \hline
      1 & 1 & 1 & 1 & 0 & 1\\
      0 & 1 & 1 & 0 & 0 & 0
    \end{array}
  \right)
  \left(
    \begin{array}{cccccc}
      1 & 0 & 0 & 0 & 0 & 0\\
      0 & 0 & 0 & 1 & 0 & 0\\
      0 & 0 & 1 & 0 & 0 & 0\\
      0 & 1 & 0 & 0 & 0 & 0\\
      0 & 0 & 0 & 0 & 1 & 0\\
      0 & 0 & 0 & 0 & 0 & 1
    \end{array}
  \right).
  \]
\end{ex}
\begin{figure}[t]
  \centering
  \includegraphics[clip,width=4.5cm]{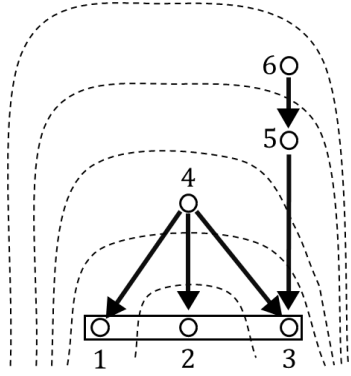}
  \caption{Construction of maximal chain of the minimizers of the MVSP corresponding to the matrix $A$ in Example \ref{ex:maximalchain}. Dots and arrows constitute the Hasse diagram. Three elements in the rectangular box are collinear. Dotted lines represent a maximal chain of consistent subspaces. }
  \label{fig:maximalChain}
\end{figure}
\subsection*{Acknowledgments}
We thank Taihei Oki for helpful comments and discussion on Theorem \ref{thm:upperBound}.
We also thank Kazuo Murota for careful reading and comments.
This research is supported by JSPS KAKENHI Grant Numbers, 25280004, 26330023, 26280004, 17K00029.

%%%%%%%%%%%%%%%%%%%%%%%%%%%%%%%%%%%%%%%%%%%%%%%%%%%%%%%%%%%%%%%%%%%
%%  BIBLIOGRAPHY                                                 %%
%%%%%%%%%%%%%%%%%%%%%%%%%%%%%%%%%%%%%%%%%%%%%%%%%%%%%%%%%%%%%%%%%%%

\end{document}